\documentclass[10 pt,onecolumn,draftcls]{IEEETran}  

\usepackage{graphicx,epstopdf,caption,subcaption}
\usepackage{amsopn}

\usepackage{subeqnarray}
\usepackage{amssymb,amsmath,amsthm,mathtools}
\usepackage{ifthen}
\usepackage{setspace}
\usepackage{cite}
\usepackage{color}
\newcommand{\R}{\mathbb{R}}
\newcommand{\qedd}{\ \hfill{$\Box$}}

\newboolean{isfullversion}
\setboolean{isfullversion}{true}

\newboolean{showcomments}
\setboolean{showcomments}{true}
\newcommand{\lina}[1]{  \ifthenelse{\boolean{showcomments}}
	{ \textcolor{red}{(  #1)}} {}  }

\newtheorem{theorem}{Theorem}
\newtheorem{lemma}{Lemma}
\newtheorem{remark}{Remark}
\newtheorem{corollary}{Corollary}
\newtheorem{proposition}{Proposition}
\newtheorem{assumption}{Assumption}
\allowdisplaybreaks
\newcommand{\vv}[1]{\boldsymbol{\mathbf{#1}}}
\title{\LARGE \bf
On the Exponential Stability of Primal-Dual Gradient Dynamics*
}

\author{Guannan Qu$^{1}$ and Na Li$^{1}$
\thanks{*The work was supported by NSF 1608509, NSF CAREER 1553407, AFOSR YIP, and ARPA-E through the NODES program.
	.}
\thanks{$^{1}$ Guannan Qu and Na Li are with John A. Paulson School of Engineering and Applied Sciences, Harvard University, Cambridge MA 02138, USA.
       Emails: {\tt\small gqu@g.harvard.edu}, {\tt\small nali@seas.harvard.edu}}%
}

\begin{document}

\maketitle
\thispagestyle{empty}
\pagestyle{empty}

\begin{abstract}
Continuous time primal-dual gradient dynamics that find a saddle point of a Lagrangian of an optimization problem have been widely used in systems and control. While the global asymptotic stability of such dynamics has been well-studied, it is less studied whether they are globally exponentially stable. In this paper, we study the primal-dual gradient dynamics for convex optimization with strongly-convex and smooth objectives and affine equality or inequality constraints, and prove global exponential stability for such dynamics. Bounds on decaying rates are provided. 
\end{abstract}


\section{INTRODUCTION}

This paper considers the following constrained optimization problem 
\begin{align} 
&\min_{\vv{x}\in \R^n} f(\vv{x}) \label{eq:intro_opt_problem} \\
\text{s.t.} \quad&  A_1\vv{x}  = \vv{b}_1 \nonumber \\
& A_2\vv{x} \leq \vv{b}_2 \nonumber
\end{align}
where $A_1\in\R^{m_1\times n},A_2\in\R^{m_2\times n} $ and $\vv{b}_1\in\R^{m_1}, \vv{b}_2\in\R^{m_2}$ and $f(\vv{x})$ is a strongly convex and smooth function. Let $L(\vv{x},\vv{\lambda})$ be the Lagrangian (or Augmented Lagrangian) associated with Problem \eqref{eq:intro_opt_problem}. 
The focus of this paper is the following primal-dual gradient dynamics, also known as saddle-point dynamics, associated with the Lagrangian $L(\vv{x},\vv{\lambda})$, 
\begin{subeqnarray} \label{eq:intro_pdgd}
\dot{\vv{x}} &=& -\eta_1 \nabla_{\vv{x}} L(\vv{x},\vv{\lambda}) \slabel{eq:intro_pdgd_p}\\
\dot{\vv{\lambda}} &=& \eta_2\nabla_{\vv{\lambda}} L(\vv{x},\vv{\lambda})\slabel{eq:intro_pdgd_d}
\end{subeqnarray}
where $\eta_1,\eta_2>0$ are time constants.

Primal-Dual Gradient Dynamics (PDGD), also known as saddle-point dynamics, were first introduced in \cite{kose1956solutions,arrow1958studies}. They have been widely used in engineering and control systems, for example in power grid \cite{zhao2014design,li2016connecting}, wireless communication \cite{chiang2007layering,chen2012convergence}, network and distributed optimization \cite{wang2011control,niederlander2016distributed}, game theory \cite{gharesifard2013distributed}, etc. Despite its wide applications, general studies on PDGD \cite{kose1956solutions,arrow1958studies,uzawa1958iterative,polyak1970iterative,golshtein1974generalized,zabotin1988subgradient,kallio1999large,kallio1994perturbation,korpelevich1976extragradient,nedic2009subgradient,feijer2010stability,holding2014convergence,cherukuri2016asymptotic,niederlander2016distributed,cherukuri2016role,cherukuri2017saddle} have mostly focused on its asymptotic stability (or convergence), with few studying its global exponential stability. It is known that the gradient dynamics for the unconstrained version of (\ref{eq:intro_opt_problem}) achieves global exponential stability when $f$ is strongly convex and smooth. It is natural to raise the question whether in the constrained case, PDGD can also achieve global exponential stability. 

Global exponential stability is a desirable property in practice. Firstly, in control systems especially those in critical infrastructure like the power grid, it is desirable to have strong stability guarantees. Secondly, when using PDGD as computational tools for constrained optimization, discretization is essential for implementation. The global exponential stability ensures that the simple explicit Euler discretization has a geometric convergence rate when the discretization step size is sufficiently small \cite{stuart1994numerical,stetter1973analysis}. This is an appealing property for discrete-time optimization methods.  

{\color{black}\emph{Contribution of this paper.} In this paper, we prove the \text{global} exponential stability of PDGD \eqref{eq:intro_pdgd} under some regularity conditions on problem~\eqref{eq:intro_opt_problem} and we also give bounds on the decaying rates (Theorem~\ref{thm:convergence} and \ref{thm:ineq}). Our proof relies on a quadratic Lyapunov function that has non-zero off-diagonal terms, which is different from the (block-)diagonal quadratic Lyapunov functions that are commonly used in the literature \cite{cherukuri2016asymptotic,chen2012convergence} and are known being unable to certify global exponential stability \cite[Lemma 3]{chen2012convergence}. 
We also highlight that when handling inequality constraints, we use a variant of the PDGD based on Augmented Lagrangian \cite{bertsekas2014constrained} and is projection free. This is different from the projection-based PDGD studied in \cite{feijer2010stability,cherukuri2016asymptotic}, which is discontinuous (see Remark~\ref{rem:projection}). 
Our variant of PDGD guarantees that the multipliers stay nonnegative without using projection, and avoids the discontinuity problem caused by projection \cite{feijer2010stability,cherukuri2016asymptotic} (see Remark \ref{rem:smooth}).
 }



\subsection{Related Work}
There have been many efforts in studying the stability of PDGD as well as its discrete time version. An incomplete list includes \cite{kose1956solutions,arrow1958studies,uzawa1958iterative,polyak1970iterative,golshtein1974generalized,zabotin1988subgradient,kallio1999large,kallio1994perturbation,korpelevich1976extragradient,nedic2009subgradient,feijer2010stability,holding2014convergence,cherukuri2016asymptotic,niederlander2016distributed,cherukuri2016role,cherukuri2017saddle}. For instance, \cite{nedic2009subgradient} studies the subgradient saddle point algorithm and proves its convergence to an approximate saddle point with rate $O(\frac{1}{t})$; \cite{feijer2010stability} uses LaSalle invariance principle to prove global asympotic stability of PDGD; \cite{cherukuri2017saddle} studies the global asymptotic stability of the saddle-point dynamics associated with general saddle functions; and \cite{cherukuri2016asymptotic} proves global asymptotic stability of PDGD with projection, which is extended in \cite{cherukuri2016role} by using a weaker assumption and proving input-to-state stability. 

Our work is closely related to a recent paper \cite{niederlander2016distributed} which studies saddle-point-like dynamics and proves global exponential stability when applying such dynamics to equality constrained convex optimization problems. The difference between \cite{niederlander2016distributed} and our work is that for the equality constrained case, \cite{niederlander2016distributed} considers a different Lagrangian from ours. Further, the result in \cite{niederlander2016distributed} cannot be directly generalized to inequality constrained case \cite[Remark 3.9]{niederlander2016distributed}. 

{\color{black}Our work is also related to the vast literature on spectral bounds on saddle matrices \cite{benzi2005numerical,shen2010eigenvalue}. Such bounds, when combined with Ostrowski Theorem \cite[10.1.4]{ortega1970iterative}, can lead to \textit{local} exponential stability results of PDGD \cite[Sec 4.4.1]{bertsekas2014constrained} \cite[Prop. 4.4.1]{bertsekas1999nonlinear} as opposed to \emph{global} exponential stability which is the focus of this paper. }

{\color{black}It recently came to our attention that \cite{dhingra2016proximal} studies a class of dynamics, a special case of which turns out to be similar to our PDGD for affine inequality constraints. \cite{dhingra2016proximal} also proves global exponential stability. Their proof uses frequency domain analysis, which is different from our time-domain analysis. It remains interesting to investigate the connection between the methods of \cite{dhingra2016proximal} and our work. }




\textbf{Notations.} Throughout the paper, scalars will be small letters, vectors will be bold small letters and matrices will be capital letters. Notation $\Vert \cdot\Vert$ represents Euclidean norm for vectors, and spectrum norm for matrices. For any symmetric matrix $P_1, P_2$ of the same dimension, $P_1\succeq P_2$ means $P_1 - P_2$ is positive semi-definite. 


\section{Algorithms and Main Results}\label{sec:summary}
In this section we describe our PDGD for solving Problem~\eqref{eq:intro_opt_problem} and present stability results. Throughout this paper, we use the following assumption of $f$:
\begin{assumption}\label{assump:f}
Function $f$ is twice differentiable, $\mu$-strongly convex and $\ell$-smooth, i.e. for all $\vv{x},\vv{y}\in\R^n$, 
	\begin{align}
	\mu \Vert \vv{x} - \vv{y}\Vert^2 \leq \langle \nabla f(\vv{x}) - \nabla f(\vv{y}) , \vv{x} - \vv{y}\rangle \leq \ell \Vert \vv{x} - \vv{y}\Vert^2
	\end{align}
\end{assumption}
 
 To streamline exposition, we will present the equality constrained case and the inequality constrained case separately. Integrating them will give PDGD with global exponential stability for Problem~\eqref{eq:intro_opt_problem}. 
  Without causing any confusion, 
  notations will be double-used in the two cases.
 
\subsection{Equality Constrained Case}

We first consider the equality constrained case,
\begin{align}
\min_{\vv{x}\in \R^n} \quad & f(\vv{x}) \label{eq:opt_problem} \\
\text{s.t.} \quad&  A\vv{x} = \vv{b} \nonumber
\end{align}
Here we remove the subscript for $A$ and $\vv{b}$ in Problem~\eqref{eq:intro_opt_problem} for notational simplicity. Problem (\ref{eq:opt_problem}) has the Lagrangian,
\begin{align}
L(\vv{x}, \vv{\lambda}) & = f(\vv{x}) + \vv{\lambda}^T(A\vv{x} - \vv{b}) \label{eq:lagrangian}
\end{align}
where $\vv{\lambda} \in \R^m$ is the Lagrangian multiplier. The PDGD is,
\begin{subeqnarray} \label{eq:pdgd_eq}
\dot{\vv{x}} &=& - \nabla_{\vv{x}} L(\vv{x},\vv{\lambda}) = -  \nabla f(\vv{x})  - A^T \vv{\lambda}  \slabel{eq:pdgd_p}\\
\dot{\vv{\lambda}} &=& \eta\nabla_\lambda L(\vv{x},\vv{\lambda} ) = \eta(A \vv{x} - \vv{b})\slabel{eq:pdgd_d}
\end{subeqnarray}
where without loss of generality, we have fixed the time constant of the primal part to be $1$. We make the following assumption on $A$, which is the linear independence constraint qualification for \eqref{eq:opt_problem}.

\begin{assumption}\label{assump:A}
	We assume that matrix $A$ is full row rank and  $\kappa_1 I \preceq AA^T \preceq \kappa_2 I$ for some $\kappa_1,\kappa_2>0$. 
\end{assumption}

Let $(\vv{x}^*,\vv{\lambda}^*)$ be the equilibrium point of (\ref{eq:pdgd_eq}), which in this case is also the saddle point of $L$.\footnote{ Assumption~\ref{assump:f} and \ref{assump:A} guarantee that the saddle point exists and is unique. } The following theorem gives the global exponential stability of the PDGD (\ref{eq:pdgd_eq}). 
\begin{theorem}\label{thm:convergence}
Under Assumption~\ref{assump:f} and \ref{assump:A},	for $\eta>0$, define $\tau_{eq} = \min (\frac{\eta \kappa_1}{4\ell  }, \frac{\kappa_1\mu}{4 \kappa_2} ) $. Then there exist constants $C_1,C_2$ that depend on $\eta,\kappa_1,\kappa_2,\mu,\ell$, $\Vert \vv{x}(0) - \vv{x}^*\Vert$, $\Vert \vv{\lambda}(0) - \vv{\lambda}^* \|$,  s.t. $\Vert \vv{x}(t) - \vv{x}^*\Vert \leq C_1 e^{- \frac{1}{2} \tau_{eq} t}$ and $\Vert \vv{\lambda}(t) - \vv{\lambda}^* \Vert \leq C_2 e^{-\frac{1}{2}\tau_{eq} t}$. 
\end{theorem}
\begin{remark}
	A result similar to Theorem~\ref{thm:convergence} is obtained in \cite[Thm. 3.6]{niederlander2016distributed} where exponential stability is proven in a similar setting when the Lagrangian is of the form $L(\vv{x},\vv{\lambda}) = f(\vv{x}) + \vv{\lambda}^T(A\vv{x} - \vv{b}) + \frac{1}{2\rho} \Vert A\vv{x} - \vv{b}\Vert^2$ for $\rho\in(0,1)$. Put in the context of \cite{niederlander2016distributed}, Theorem~\ref{thm:convergence} corresponds to the $\rho\rightarrow\infty$ limit, and can serve as a complementary result of \cite[Thm. 3.6]{niederlander2016distributed}.
\end{remark}
%
\vspace{-4pt}
\subsection{Inequality Constrained Case}\label{subsec:ineq:problem_formulation}
Now we consider the inequality constrained case,
\begin{align}
\min_{\vv{x}\in \R^n} \quad & f(\vv{x}) \label{eq:ineq:opt_problem} \\
\text{s.t.} \quad&  A\vv{x} \leq \vv{b} \nonumber
\end{align}
where $f$ and $A$ satisfy Assumption~\ref{assump:f} and \ref{assump:A}. For the inequality constrained case, we use the ``Augmented Lagrangian'' \cite[Sec. 3.1]{bertsekas2014constrained}, as opposed to the standard Lagrangian in \cite{feijer2010stability,cherukuri2016asymptotic}. In details, let $A^T = [\vv{a}_1,\vv{a}_2,\ldots, \vv{a}_m]$, with each $\vv{a}_j\in\R^n$, and let $\vv{b} = [b_1,\ldots,b_m]^T$. Then we define the augmented Lagrangian,
{\small\begin{align}
L  (\vv{x},\vv{\lambda})  = f(\vv{x} ) + \sum_{j=1}^m H_\rho(\vv{a}_j^T \vv{x}- b_j, \lambda_j ) \label{eq:ineq:aug_lagrangian}
\end{align}}where $\rho>0$ is a free parameter, $H_\rho(\cdot,\cdot):\R^2\rightarrow \R$ is a penalty function on constraint violation, defined as follows 
{\small\begin{align*}
&H_\rho(\vv{a}_j^T \vv{x}- b_j , \lambda_j) \\
&= \left\{ \begin{array}{ll}
(\vv{a}_j^T \vv{x}- b_j ) \lambda_j & \\
\qquad + \frac{\rho}{2}(\vv{a}_j^T \vv{x}- b_j)^2  & \text{if } \rho(\vv{a}_j^T \vv{x}- b_j) + \lambda_j\geq 0\\
-\frac{1}{2} \frac{\lambda_j^2}{\rho}& \text{if } \rho(\vv{a}_j^T \vv{x}- b_j) + \lambda_j < 0
\end{array}  \right.
\end{align*}}We can then calculate the gradient of $H_\rho$ w.r.t. $\vv{x}$ and $\vv{\lambda}$.
{\small\begin{align*}
&\nabla_{\vv{x}} H_\rho(\vv{a}_j^T\vv{x} - b_j , \lambda_j)= \max(\rho(\vv{a}_j^T\vv{x} - b_j) + \lambda_j,0) \vv{a}_j \\
&\nabla_{\vv{\lambda}} H_\rho(\vv{a}_j^T\vv{x} - b_j , \lambda_j) = \frac{ \max(\rho(\vv{a}_j^T\vv{x} - b_j) + \lambda_j,0) - \lambda_j}{\rho} \vv{e}_j 
\end{align*}}where $\vv{e}_j \in\R^m$ is a vector with the $j$'th entry being $1$ and other entries being $0$. The primal-dual gradient dynamics for the augmented Lagrangian $L $ is given in \eqref{eq:ineq:pdgd}. We call it as Aug-PDGD (Augmented Primal-Dual Gradient Dynamics). 
{\small\begin{subeqnarray}\label{eq:ineq:pdgd}
\dot{\vv{x}} &=& - \nabla_{\vv{x}}L  (\vv{x},\vv{\lambda}) = -\nabla f(\vv{x}) - \sum_{j=1}^m \nabla_{\vv{x}} H_\rho(\vv{a}_j^T\vv{x} - b_j , \lambda_j) \nonumber\\
&=&-\nabla f(\vv{x}) - \sum_{j=1}^m \max(\rho(\vv{a}_j^T\vv{x} - b_j) + \lambda_j,0) \vv{a}_j \slabel{eq:ineq:pdgd_p}\\
\dot{\vv{\lambda}} &= &\eta \nabla_{\vv{\lambda}}L  (\vv{x},\vv{\lambda}) = \eta \sum_{j=1}^{m} \nabla_{\vv{\lambda}} H_\rho(\vv{a}_j^T\vv{x} - b_j , \lambda_j)\nonumber\\\
&=& \eta \sum_{j=1}^{m} \frac{ \max(\rho(\vv{a}_j^T\vv{x} - b_j) + \lambda_j,0) - \lambda_j}{\rho} \vv{e}_j  \slabel{eq:ineq:pdgd_d}
\end{subeqnarray}}

\begin{remark}\label{rem:projection}
	The Lagrangian \eqref{eq:ineq:aug_lagrangian} we use is different from the standard Lagrangian used in \cite{feijer2010stability,cherukuri2016asymptotic}. The standard Lagrangian and the associated PDGD in \cite{feijer2010stability,cherukuri2016asymptotic} involves a discontinuous projection step, which creates difficulties both in theoretic analysis and numerical simulations. Theoretically, the projection step is based on Euclidean norm, and is consistent with the block diagonal Lyapunov function used in \cite{feijer2010stability,cherukuri2016asymptotic}, but it is not consistent with the Lyapunov function with cross term used in this paper (cf. \eqref{eq:ineq:P}), which is the key for proving global exponential stability. Therefore we conjecture that the PDGD with projection \cite{feijer2010stability,cherukuri2016asymptotic} is \textit{not} exponentially stable. Numerically, when we simulate the PDGD with the discontinuous projection step using MATLAB ODE solvers, we encounter many numerical issues. For these reasons, in this paper we study an alternative projection-free PDGD based on the Augmented Lagrangian.
\end{remark}
\begin{remark}\label{rem:smooth}
	It is easy to check that, if $\lambda_j(0)\geq 0$, then \eqref{eq:ineq:pdgd_d} guarantees $\lambda_j(t) \geq 0, \forall t$. This means that the dynamics \eqref{eq:ineq:pdgd} automatically guarantees $\lambda_j(t)$ will stay nonnegative as long as its initial value is nonnegative, without using projection as is done in \cite{feijer2010stability,cherukuri2016asymptotic}, thus avoiding discontinuity issues caused by the projection step.    
\end{remark}
Since the saddle point of the Augmented Lagrangian \eqref{eq:ineq:aug_lagrangian} is the same as that of the standard Lagrangian (see \cite[Sec. 3.1]{bertsekas2014constrained} for details), we have the following proposition regarding the equilibrium point $(\vv{x}^*,\vv{\lambda}^*)$ of Aug-PDGD. For completeness we include a proof
in \ifthenelse{\boolean{isfullversion}}{Appendix-\ref{appendix:ineq:fixedpoint}}{our online report \cite[Appendix-E]{fullversion}}. 
\begin{proposition}[\cite{{bertsekas2014constrained}}]\label{thm:ineq:fixedpoint}
	Under Assumption~\ref{assump:f} and \ref{assump:A}, Aug-PDGD (\ref{eq:ineq:pdgd}) has a unique equilibrium point $(\vv{x}^*,\vv{\lambda}^*)$ and it satisfies the KKT condition of problem (\ref{eq:ineq:opt_problem}).
\end{proposition}

Aug-PDGD \eqref{eq:ineq:pdgd} is globally exponentially stable, as stated below.

\begin{theorem}\label{thm:ineq}
	Under Assumption~\ref{assump:f} and \ref{assump:A}, the Aug-PDGD \eqref{eq:ineq:pdgd} is globally exponentially stable in the sense that, for any $\eta>0$, $\rho>0$, there exists constant {\small$$\tau_{ineq} = \frac{\eta\kappa_1^2}{40 \ell \kappa_2 [ \max(\frac{\rho \kappa_2}{\mu}, \frac{\ell}{\mu})]^2 [\max(\frac{\eta}{\ell \rho}, \frac{\ell}{\mu})]^2 }$$}and constants $C_3,C_4>0$ which depend on $\eta,\rho,\kappa_1,\kappa_2,\mu,\ell$, $\Vert \vv{x}(0) - \vv{x}^*\Vert$, $\Vert \vv{\lambda}(0) - \vv{\lambda}^* \|$, s.t. $\Vert \vv{x}(t) - \vv{x}^*\Vert\leq C_3 e^{-\frac{1}{2} \tau_{ineq} t}$, $\Vert \vv{\lambda}(t) - \vv{\lambda}^*\Vert\leq C_4e^{-\frac{1}{2} \tau_{ineq} t}$.
\end{theorem}

\begin{remark} In this section, we only study the affine inequality constrained case and assume the matrix $A$ satisfies Assumption~\ref{assump:A}. \ifthenelse{\boolean{isfullversion}}{In Appendix-\ref{appendix:rank}}{In our online report \cite[Appendix-I]{fullversion}}, we will extend our results by relaxing Assumption 2 to be the linear independence constraint qualification, i.e. at the optimizer $\vv{x}^*$, the submatrix of $A$ associated with the active constraints has full row rank. \ifthenelse{\boolean{isfullversion}}{For more details please see Appendix-\ref{appendix:rank}.}{For more details, please see \cite[Appendix-I]{fullversion}.} Beyond the affine case, for nonlinear convex constraint $\vv{J}(x)\leq 0$  where $\vv{J}:\R^n\rightarrow\R^m$, we conjecture that exponential stability still holds and we need to replace Assumption~\ref{assump:A} with the condition $ \kappa_1 I \preceq \frac{\partial{ \vv{J}(\vv{x})}}{\partial \vv{x}}  (\frac{\partial{ \vv{J}(\vv{x})}}{\partial \vv{x}} )^T \preceq \kappa_2 I$ where $\frac{\partial{ \vv{J}(\vv{x})}}{\partial \vv{x}}\in\R^{m\times n}$ is the Jacobian of $\vv{J}$ w.r.t. $\vv{x}$. We leave the conjecture to our future work.
\end{remark}


\section{Stability Analysis}\label{sec:analysis}
In this section, we prove global exponential stability.
We also show global exponential stability ensures the geometric convergence rate of the Euler discretization. 
\subsection{The Equality Constrained Case, Proof of Theorem~\ref{thm:convergence}}    \label{subsec:eq}
{\color{black} We stack $\vv{x}$ and $\vv{\lambda}$ into a larger vector $\vv{z} = [\vv{x}^T, \vv{\lambda}^T]^T$ and similarly define $\vv{z}^* = [(\vv{x}^*)^T, (\vv{\lambda}^*)^T]^T$. We define quadratic Lyapunov function, $V(\vv{z}) = (\vv{z} - \vv{z}^*)^T P (\vv{z} - \vv{z}^*)$ with $P\succ 0$ defined by
\begin{align}
P = \left[\begin{array}{cc}
\eta c I & \eta A^T\\
\eta A & c I
\end{array}  \right] \in \R^{(m+n)\times (m+n)}\label{eq:P}
\end{align} 
where $c = 4 \max(\ell ,\frac{\eta \kappa_2}{\mu} )$.\footnote{We have $P\succ 0$ as long as $c^2> \eta \kappa_2$, which is met by our choice of $c$. } If we can show the following property of $V(\vv{z})$ along the trajectory of the dynamics,
\begin{align}
\frac{d}{dt} V(\vv{z}) \leq -\tau V(\vv{z}) \label{eq:dv}
\end{align}
for $ \tau =  \frac{\eta \kappa_1}{c} = \min (\frac{\eta \kappa_1}{4\ell }, \frac{\kappa_1\mu}{4 \kappa_2} )$, then we have proved Theorem~\ref{thm:convergence}. 
The rest of the section will be devoted to proving \eqref{eq:dv}. We start with the following auxiliary Lemma, which can be proved by using mean value theorem. A similar lemma can be found in \cite[Lem. A.1]{niederlander2016distributed}, and for completeness we include a proof in \ifthenelse{\boolean{isfullversion}}{Appendix-\ref{appendix:proof_B}. }{our online report \cite[Appendix-D]{fullversion}.}

\begin{lemma}\label{lem:B}
	Under Assumption 1, for any $\vv{x}\in\R^n$, there exists a symmetric matrix $B(\vv{x})$ that depends on $\vv{x}$, satisfying $\mu I \preceq B(\vv{x}) \preceq \ell I $, s.t. $\nabla f(\vv{x}) - \nabla f(\vv{x}^*) = B({\vv{x}}) (\vv{x} - \vv{x}^* )$.
\end{lemma}

With Lemma~\ref{lem:B}, we can rewrite PDGD (\ref{eq:pdgd_eq}) as,
{\small	\begin{align}
	\frac{d}{dt} {\vv{z}} & = \left[ \begin{array}{c}
 -( \nabla_{\vv{x}} L (\vv{x},\vv{\lambda}) - \nabla_{\vv{x}} L(\vv{x}^*,\vv{\lambda}^*))\\
 \eta \nabla_{\vv{\lambda}} L(\vv{x},\vv{\lambda}) - \eta \nabla_{\vv{\lambda}} L(\vv{x}^*,\vv{\lambda}^*)
	\end{array}  \right] \nonumber \\
	&=  \left[ \begin{array}{c}
	- B(\vv{x}) (\vv{x} - \vv{x}^*)  - A^T(\vv{\lambda}-\vv{\lambda}^* ) \\
	\eta A(\vv{x} - \vv{x}^*)
	\end{array}  \right]\nonumber\\
	&= \left[ \begin{array}{cc}
	- B(\vv{x})  &  - A^T  \\
	\eta A & 0
	\end{array}  \right]  (\vv{z}-\vv{z}^*)  \coloneqq G(\vv{z})(\vv{z} - \vv{z}^*). \label{eq:pdgd_linear}
	\end{align} }
Then, $\frac{d}{dt} V(\vv{z}) $ can be written as 
{\small\begin{align}
	&\frac{d}{dt} V(\vv{z})=  \dot{\vv{z}}^T P ( \vv{z}-\vv{z}^*) + (\vv{z}-\vv{z}^*)^T P \dot{\vv{z}}\nonumber \\
	&=(\vv{z}-\vv{z}^*) ^T(G(\vv{z})^TP + PG(\vv{z})   )(\vv{z} - \vv{z}^*) \label{eq:dvgp}
	\end{align}}
Therefore, to prove \eqref{eq:dv}, it is sufficient to prove the following Lemma, whose proof is in Appendix-\ref{appendix:qp}.
\begin{lemma}\label{lem:qp}
For any $\vv{z}\in\R^{n+m}$, we have $$G(\vv{z})^TP + PG(\vv{z})  \preceq -\tau P \label{eq:gp}$$
\end{lemma}
Lemma \ref{lem:qp} and \eqref{eq:dvgp} lead to \eqref{eq:dv}, concluding the proof. 
}

\subsection{The Inequality Constrained Case, Proof of Theorem~\ref{thm:ineq}} \label{subsec:ineq}
{{\color{black}  We start by emphasizing the notations in this section is independent from the equality constrained case in Section \ref{subsec:eq}. We stack $\vv{x}$, $\vv{\lambda}$ into a larger vector $\vv{z} = [\vv{x}^T,\vv{\lambda}^T]^T$ and similarly define $\vv{z}^* = [(\vv{x}^*)^T, (\vv{\lambda}^*)^T]^T$. Next, we define the following quadratic Lyapunov function $V(\vv{z}) = (\vv{z} - \vv{z}^*)^T P (\vv{z} - \vv{z}^*)$ with $P\succ 0$ defined by,
\begin{align}
P = \left[\begin{array}{cc}
\eta c I & \eta A^T\\
\eta A & c I
\end{array}  \right] \in \R^{(m+n)\times (m+n)}\label{eq:ineq:P}
\end{align} 
where $c = 20 \ell [ \max(\frac{\rho \kappa_2}{\mu}, \frac{\ell}{\mu})]^2 [\max(\frac{\eta}{\ell \rho}, \frac{\ell}{\mu})]^2 \frac{\kappa_2}{\kappa_1}$.\footnote{We have $P\succ 0$ as long as $c^2> \eta \kappa_2$, which is met by our choice of $c$. } Then, the results of Theorem~\ref{thm:ineq} directly follows from the following property of $V(\vv{z})$, 
\begin{align}
\frac{d}{dt} V(\vv{z}) \leq -\tau V(\vv{z}) \label{eq:ineq:dv}
\end{align}
where $\tau = \frac{\eta \kappa_1}{2c} = \frac{\eta\kappa_1^2}{40 \ell \kappa_2 [ \max(\frac{\rho \kappa_2}{\mu}, \frac{\ell}{\mu})]^2 [\max(\frac{\eta}{\ell \rho}, \frac{\ell}{\mu})]^2 }$. The rest of the section will be devoted to proving \eqref{eq:ineq:dv}. 
To prove \eqref{eq:ineq:dv}, we write the Aug-PDGD \eqref{eq:ineq:pdgd} in a ``linear'' form. In addition to Lemma~\ref{lem:B} we need the following Lemma. 
\begin{lemma}\label{lem:ineq:gamma_def}
For any $j$ and $\vv{z} = [\vv{x}^T,\vv{\lambda}^T]^T\in\R^{n+m}$, we have there exists $\gamma_j(\vv{z})\in[0,1]$ that depends on $\vv{z}$ s.t. 
\vspace{-6pt}
{\small\begin{align*}
&\nabla_{\vv{x}} H_\rho(\vv{a}_j^T\vv{x} - b_j , \lambda_j) - \nabla_{\vv{x}} H_\rho(\vv{a}_j^T\vv{x}^* - b_j , \lambda_j^*)\\
& = \gamma_j(\vv{z})  \rho \vv{a}_j^T(\vv{x}- \vv{x}^*) \vv{a}_j+ \gamma_j(\vv{z})( \lambda_j- \lambda^*_j)\vv{a}_j  \\
&\nabla_{\vv{\lambda}} H_\rho(\vv{a}_j^T\vv{x} - b_j , \lambda_j) - \nabla_{\vv{\lambda}} H_\rho(\vv{a}_j^T\vv{x}^* - b_j , \lambda_j^*) \\
&= \gamma_j (\vv{z})\vv{a}_j^T(\vv{x} - \vv{x}^*) \vv{e}_j + \frac{1}{\rho} (\gamma_j (\vv{z}) - 1) (\lambda_j - \lambda^*_j) \vv{e}_j
\vspace{-6pt}
\end{align*}}
\end{lemma}

\begin{proof}
The lemma directly follows from that for any $y, y^*\in\R$, there exists some $\gamma\in[0,1]$, depending on $y,y^*$ s.t. $\max(y,0) - \max(y^*,0) = \gamma (y - y^*)$. To see this, when $y \neq y^*$, set $\gamma = \frac{\max(y,0) - \max(y^*,0) }{y - y^*} $; otherwise, set $\gamma= 0$.  	
\end{proof}

For any $\vv{z}\in\R^{n+m} $, we define notation $\Gamma(\vv{z}) = \text{diag}(\gamma_1(\vv{z}),\ldots,\gamma_m(\vv{z}))\in\R^{m\times m}$, where $\gamma_j(\vv{z})$ is from Lemma~\ref{lem:ineq:gamma_def}. 
With notation $\Gamma(\vv{z})$, we can then rewrite the Aug-PDGD \eqref{eq:ineq:pdgd_p} as
{\small
\begin{align*}
\dot{\vv{x}} &=- ( \nabla_{\vv{x}} L (\vv{x},\vv{\lambda}) - \nabla_{\vv{x}} L  (\vv{x}^*, \vv{\lambda}^*) ) \\
&=  - (\nabla f(\vv{x}) - \nabla f(\vv{x}^*) ) \\
&\quad -  \sum_{j=1}^m \Bigg[\nabla_{\vv{x}} H_\rho(\vv{a}_j^T\vv{x} - b_j , \lambda_j) - \nabla_{\vv{x}} H_\rho(\vv{a}_j^T\vv{x}^* - b_j , \lambda_j^*) \Bigg]\\
&= - B(\vv{x})(\vv{x} - \vv{x}^*) - \rho A^T\Gamma(\vv{z}) A(\vv{x}- \vv{x}^*) - A^T \Gamma(\vv{z})(\vv{\lambda} - \vv{\lambda}^*) 
\end{align*} }where $\mu I \preceq B(\vv{x}) \preceq \ell I$ (Lemma~\ref{lem:B}). We then rewrite \eqref{eq:ineq:pdgd_d}, 
{\small\begin{align*}
\dot{\vv{\lambda}} &= \eta \nabla_{\vv{\lambda}} L (\vv{x},\vv{\lambda}) - \eta \nabla_{\vv{\lambda}} L (\vv{x}^*,\vv{\lambda}^*) \\
&= \eta \sum_{j=1}^m\Big[ \nabla_{\vv{\lambda}} H_\rho(\vv{a}_j^T\vv{x} - b_j,\lambda_j ) - \nabla_{\vv{\lambda}} H_\rho(a_j^T\vv{x}^* - b_j, {\lambda}_j^* )\Big]\\
&= \eta \Gamma(\vv{z}) A (\vv{x}-\vv{x}^*) + \frac{\eta}{\rho} (\Gamma(\vv{z}) - I) ( \vv{\lambda} -\vv{\lambda}^*)
\end{align*}}Then, the Aug-PDGD \eqref{eq:ineq:pdgd} can be written as,
{\small\begin{align}
\dot{\vv{z}} & = \left[ \begin{array}{cc}
	- B(\vv{x}) - \rho A^T\Gamma(\vv{z}) A &  - A^T \Gamma(\vv{z}) \\
	\eta \Gamma(\vv{z}) A & \frac{ \eta}{\rho}(\Gamma(\vv{z}) - I)
	\end{array}  \right] (\vv{z} - \vv{z}^*) \nonumber \\
	&\coloneqq  G(\vv{z})({\vv{z}} -\vv{z}^*). \label{eq:ineq:linear}
\end{align}}Then, $\frac{d}{dt} V(\vv{z} ) $ can be written as 
{\small\begin{align}
&\frac{d}{dt} V(\vv{z})=  \dot{\vv{z}}^T P ( \vv{z}-\vv{z}^*) + (\vv{z}-\vv{z}^*)^T P \dot{\vv{z}}\nonumber \\
&=(\vv{z}-\vv{z}^*) ^T(G(\vv{z})^TP + PG(\vv{z}))(\vv{z} - \vv{z}^*) \label{eq:ineq:dvgp}
\end{align}}
Therefore, to prove \eqref{eq:ineq:dv}, it is sufficient to show the following Lemma, whose proof is in Appendix-\ref{appendix:ineq:qp}.
\begin{lemma}\label{lem:ineq:qp}
	For any $\vv{z}\in\R^{n+m}$, we have
	$$G(\vv{z})^TP + PG(\vv{z})  \preceq -\tau P \label{eq:ineq:gp}$$
\end{lemma}
Lemma \ref{lem:ineq:qp} and \eqref{eq:ineq:dvgp} lead to \eqref{eq:ineq:dv}, concluding the proof.}


 
 \subsection{Discrete Time Primal-Dual Gradient Algorithm}\label{sec:discrete}
 Lastly, we briefly discuss the stability of the discretization of (Aug-)PDGD. It is known that the Euler discretization of an exponentially stable dynamical system possesses geometric convergence speed \cite{stuart1994numerical,stetter1973analysis}, provided the discretization step size is small enough. For completeness, we provide the following Lemma~\ref{lem:discretization}, whose proof can be found in \ifthenelse{\boolean{isfullversion}}{Appendix~\ref{appendix:discretization}.}{our online report \cite[Appendix-F]{fullversion}. }
     
 
 \begin{lemma}\label{lem:discretization}
 	Consider a continuous-time dynamical system $\dot{\vv{z}} = F (\vv{z}) $ where $F$ is $\nu$-Lipschitz continuous. Suppose $\vv{z}^*$ is an equilibrium point and there exists positive definite matrix $P$, constant $\tau>0$, and Lyapunov function $V(\tilde{\vv{z}}) = \tilde{\vv{z}}^T P \tilde{\vv{z}} $ such that $\frac{d}{dt} V(\tilde{\vv{z}}) \leq -\tau V(\tilde{\vv{z}}) $, where $\tilde{\vv{z}} := \vv{z} - \vv{z}^*$. Then its Euler discretization with step size $\delta>0$, 
 	$$\vv{y}(k+1) = \vv{y}(k) + \delta F(\vv{y}(k)) $$
 	satisfies $\Vert \vv{y}(k) - \vv{z}^*\Vert\leq C (e^{-\frac{\tau\delta}{2}} + \frac{\kappa_P \nu^2   \delta^2}{2})^k $, where $\kappa_P$ is the condition number of matrix $P$, and $C>0$ is some constant that depends on $\nu,\tau,P$ and $\Vert \vv{y}(0) - \vv{z}^*\Vert$. Further, $e^{-\frac{\tau\delta}{2}} + \frac{\kappa_P \nu^2   \delta^2}{2} <1$ for small enough $\delta$.
 \end{lemma}
 
 Based on the proof in Section \ref{subsec:eq} and \ref{subsec:ineq}, both PDGD \eqref{eq:pdgd_eq} and Aug-PDGD \eqref{eq:ineq:pdgd} satisfy the conditions in Lemma~\ref{lem:discretization}. 
 Hence the discretized versions converge geometrically fast for small enough discretization step size $\delta$.

 \section{Illustrative Examples} \label{sec:examples}
 \subsection{Equality Constrained Case}

We numerically study PDGD with affine equality constraints and quadratic cost functions. We let $n = 5$, $m =2$, $f(\vv{x}) = \frac{1}{2} \vv{x}^T W \vv{x}$, where $W = 10 I + W_0W_0^T$, and $W_0$ is a $n$-by-$n$ Gaussian random matrix. $A$ and $\vv{b}$ are also Gaussian random matrices (or vectors). 
 Since the cost is quadratic, the PDGD \eqref{eq:pdgd_eq} becomes an Linear Time-Invariant (LTI) system and we can determine the PDGD decaying rate by numerically calculating the eigenvalues of the resulting LTI system. We plot the decaying rate as a function of $\eta$ in the upper plot of Fig.~\ref{fig:equality}. We also simulate the PDGD for a selected number of $\eta$'s, and plot the distance to equilibrium point as a function of time in the lower plot of Fig.~\ref{fig:equality}. In both plots, we observe that increasing $\eta$ beyond a certain threshold does not lead to faster decaying rate, an interesting phenomenon that may be worth further studying.
 
 \begin{figure}
 	\begin{center}
 		 	\includegraphics[scale = 0.18]{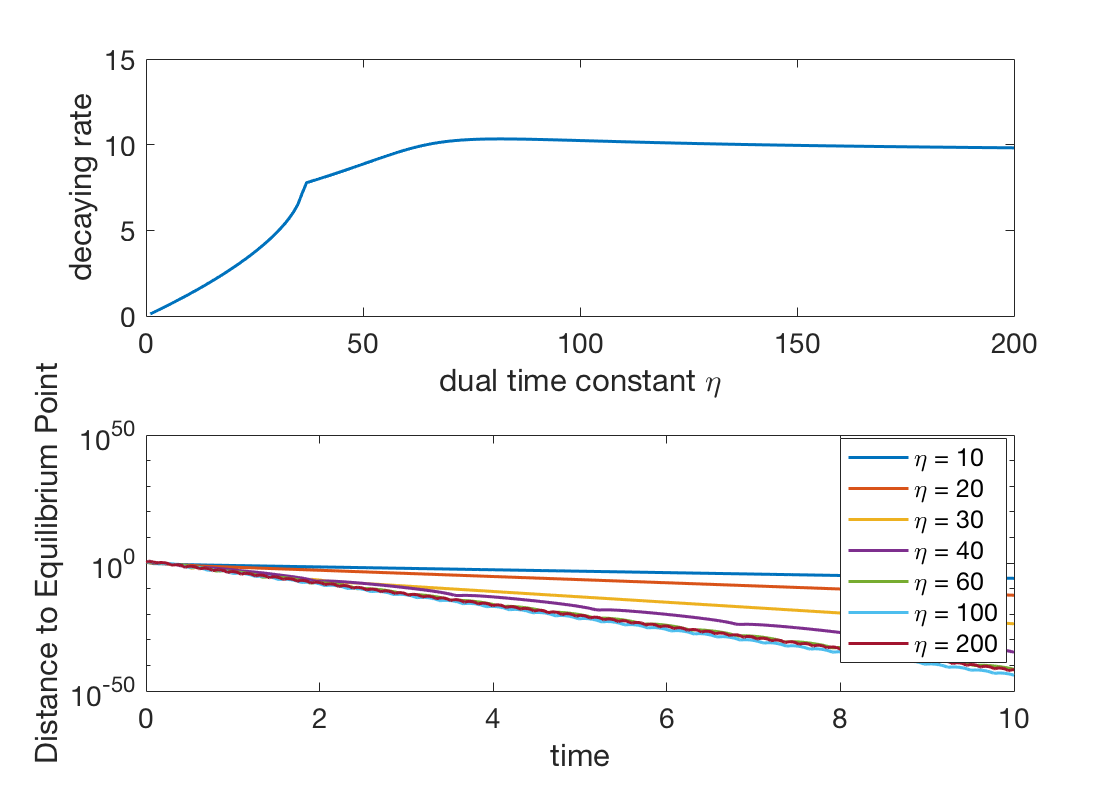}
 	\end{center}
 	\caption{Simulation of PDGD for the equality constrained case. Upper: decaying rate as a function of $\eta$. Lower: distance to the equilibrium point for some values of $\eta$. }\label{fig:equality}
 \end{figure}

\subsection{Inequality Constrained Case}
{\color{black}We numerically run the Aug-PDGD on a problem of size $n=50,m=40$. We use the loss for logistic regression \cite{logit_regression} (with synthetic data) as our cost function $f$. For the affine inequality constraint $A\vv{x}\leq \vv{b}$, every entry of $A$, $\vv{b}$ is generated independently from standard normal distribution. We fix $\rho = 1$ but try different $\eta$'s, and show the results in Fig.~\ref{fig:inequality}. Similar to the equality constrained case, here we observe that when $\eta$ is large, the decaying rate doesn't increase with $\eta$. }
\begin{figure}
	\begin{center}
	\includegraphics[scale = 0.43]{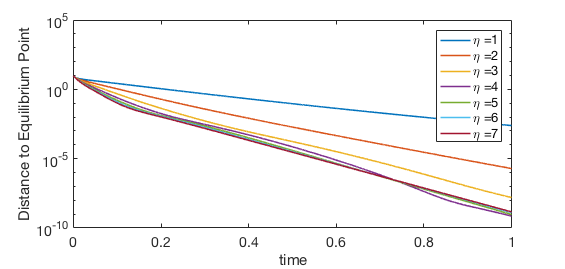}
	\end{center}
	\caption{Simulation of Aug-PDGD for the inequality constrained case under different values of $\eta$.}\label{fig:inequality}
\end{figure}
%
%
\section{Conclusions}\label{sec:conclusion}

In this paper, we study the primal-dual gradient dynamics for optimization with strongly convex and smooth objective and affine equality or inequality constraints. We prove the global exponential stability of PDGD and give explicit bounds on the decaying rates. 
Future work include 1) providing tighter bounds on the decaying rates, especially for the inequality constrained case; 2) relaxing Assumption~\ref{assump:A} for the inequality constrained case.






\bibliographystyle{IEEETran}
\bibliography{bib_primal_dual}

\begin{thebibliography}{10}
\providecommand{\url}[1]{#1}
\csname url@rmstyle\endcsname
\providecommand{\newblock}{\relax}
\providecommand{\bibinfo}[2]{#2}
\providecommand\BIBentrySTDinterwordspacing{\spaceskip=0pt\relax}
\providecommand\BIBentryALTinterwordstretchfactor{4}
\providecommand\BIBentryALTinterwordspacing{\spaceskip=\fontdimen2\font plus
\BIBentryALTinterwordstretchfactor\fontdimen3\font minus
  \fontdimen4\font\relax}
\providecommand\BIBforeignlanguage[2]{{%
\expandafter\ifx\csname l@#1\endcsname\relax
\typeout{** WARNING: IEEEtran.bst: No hyphenation pattern has been}%
\typeout{** loaded for the language `#1'. Using the pattern for}%
\typeout{** the default language instead.}%
\else
\language=\csname l@#1\endcsname
\fi
#2}}

\bibitem{kose1956solutions}
T.~Kose, ``Solutions of saddle value problems by differential equations,''
  \emph{Econometrica, Journal of the Econometric Society}, pp. 59--70, 1956.

\bibitem{arrow1958studies}
K.~J. Arrow, L.~Hurwicz, H.~Uzawa, and H.~B. Chenery, \emph{Studies in linear
  and non-linear programming}.\hskip 1em plus 0.5em minus 0.4em\relax Stanford
  University Press, 1958.

\bibitem{zhao2014design}
C.~Zhao, U.~Topcu, N.~Li, and S.~Low, ``Design and stability of load-side
  primary frequency control in power systems,'' \emph{IEEE Transactions on
  Automatic Control}, vol.~59, no.~5, pp. 1177--1189, 2014.

\bibitem{li2016connecting}
N.~Li, C.~Zhao, and L.~Chen, ``Connecting automatic generation control and
  economic dispatch from an optimization view,'' \emph{IEEE Transactions on
  Control of Network Systems}, vol.~3, no.~3, pp. 254--264, 2016.

\bibitem{chiang2007layering}
M.~Chiang, S.~H. Low, A.~R. Calderbank, and J.~C. Doyle, ``Layering as
  optimization decomposition: A mathematical theory of network architectures,''
  \emph{Proceedings of the IEEE}, vol.~95, no.~1, pp. 255--312, 2007.

\bibitem{chen2012convergence}
J.~Chen and V.~K. Lau, ``Convergence analysis of saddle point problems in time
  varying wireless systems—control theoretical approach,'' \emph{IEEE
  Transactions on Signal Processing}, vol.~60, no.~1, pp. 443--452, 2012.

\bibitem{wang2011control}
J.~Wang and N.~Elia, ``A control perspective for centralized and distributed
  convex optimization,'' in \emph{Decision and Control and European Control
  Conference (CDC-ECC), 2011 50th IEEE Conference on}.\hskip 1em plus 0.5em
  minus 0.4em\relax IEEE, 2011, pp. 3800--3805.

\bibitem{niederlander2016distributed}
S.~K. Niederl{\"a}nder and J.~Cort{\'e}s, ``Distributed coordination for
  nonsmooth convex optimization via saddle-point dynamics,'' \emph{arXiv
  preprint arXiv:1606.09298}, 2016.

\bibitem{gharesifard2013distributed}
B.~Gharesifard and J.~Cort{\'e}s, ``Distributed convergence to nash equilibria
  in two-network zero-sum games,'' \emph{Automatica}, vol.~49, no.~6, pp.
  1683--1692, 2013.

\bibitem{uzawa1958iterative}
H.~Uzawa, ``Iterative methods in concave programming,'' \emph{Studies in linear
  and non-linear programming}, 1958.

\bibitem{polyak1970iterative}
B.~Polyak, ``Iterative methods using lagrange multipliers for solving extremal
  problems with constraints of the equation type,'' \emph{USSR Computational
  Mathematics and Mathematical Physics}, vol.~10, no.~5, pp. 42--52, 1970.

\bibitem{golshtein1974generalized}
E.~Golshtein, ``Generalized gradient method for finding saddlepoints,''
  \emph{Matekon}, vol.~10, no.~3, pp. 36--52, 1974.

\bibitem{zabotin1988subgradient}
I.~Zabotin, ``A subgradient method for finding a saddle point of a
  convex-concave function,'' \emph{Issled. Prikl. Mat}, vol.~15, pp. 6--12,
  1988.

\bibitem{kallio1999large}
M.~Kallio and C.~H. Rosa, ``Large-scale convex optimization via saddle point
  computation,'' \emph{Operations Research}, vol.~47, no.~1, pp. 93--101, 1999.

\bibitem{kallio1994perturbation}
M.~Kallio and A.~Ruszczynski, ``Perturbation methods for saddle point
  computation,'' 1994.

\bibitem{korpelevich1976extragradient}
G.~Korpelevich, ``The extragradient method for finding saddle points and other
  problems,'' \emph{Matecon}, vol.~12, pp. 747--756, 1976.

\bibitem{nedic2009subgradient}
A.~Nedi{\'c} and A.~Ozdaglar, ``Subgradient methods for saddle-point
  problems,'' \emph{Journal of optimization theory and applications}, vol. 142,
  no.~1, pp. 205--228, 2009.

\bibitem{feijer2010stability}
D.~Feijer and F.~Paganini, ``Stability of primal--dual gradient dynamics and
  applications to network optimization,'' \emph{Automatica}, vol.~46, no.~12,
  pp. 1974--1981, 2010.

\bibitem{holding2014convergence}
T.~Holding and I.~Lestas, ``On the convergence to saddle points of
  concave-convex functions, the gradient method and emergence of
  oscillations,'' in \emph{Decision and Control (CDC), 2014 IEEE 53rd Annual
  Conference on}.\hskip 1em plus 0.5em minus 0.4em\relax IEEE, 2014, pp.
  1143--1148.

\bibitem{cherukuri2016asymptotic}
A.~Cherukuri, E.~Mallada, and J.~Cort{\'e}s, ``Asymptotic convergence of
  constrained primal--dual dynamics,'' \emph{Systems \& Control Letters},
  vol.~87, pp. 10--15, 2016.

\bibitem{cherukuri2016role}
A.~Cherukuri, E.~Mallada, S.~Low, and J.~Cort{\'e}s, ``The role of convexity on
  saddle-point dynamics: Lyapunov function and robustness,'' \emph{IEEE
  Transactions on Automatic Control}, 2017.

\bibitem{cherukuri2017saddle}
A.~Cherukuri, B.~Gharesifard, and J.~Cortes, ``Saddle-point dynamics:
  conditions for asymptotic stability of saddle points,'' \emph{SIAM Journal on
  Control and Optimization}, vol.~55, no.~1, pp. 486--511, 2017.

\bibitem{stuart1994numerical}
A.~M. Stuart, ``Numerical analysis of dynamical systems,'' \emph{Acta
  numerica}, vol.~3, pp. 467--572, 1994.

\bibitem{stetter1973analysis}
H.~J. Stetter, \emph{Analysis of discretization methods for ordinary
  differential equations}.\hskip 1em plus 0.5em minus 0.4em\relax Springer,
  1973, vol.~23.

\bibitem{bertsekas2014constrained}
D.~P. Bertsekas, \emph{Constrained optimization and Lagrange multiplier
  methods}.\hskip 1em plus 0.5em minus 0.4em\relax Academic press, 2014.

\bibitem{benzi2005numerical}
M.~Benzi, G.~H. Golub, and J.~Liesen, ``Numerical solution of saddle point
  problems,'' \emph{Acta numerica}, vol.~14, pp. 1--137, 2005.

\bibitem{shen2010eigenvalue}
S.-Q. Shen, T.-Z. Huang, and J.~Yu, ``Eigenvalue estimates for preconditioned
  nonsymmetric saddle point matrices,'' \emph{SIAM Journal on Matrix Analysis
  and Applications}, vol.~31, no.~5, pp. 2453--2476, 2010.

\bibitem{ortega1970iterative}
J.~M. Ortega and W.~C. Rheinboldt, \emph{Iterative solution of nonlinear
  equations in several variables}.\hskip 1em plus 0.5em minus 0.4em\relax Siam,
  1970, vol.~30.

\bibitem{bertsekas1999nonlinear}
D.~P. Bertsekas, \emph{Nonlinear programming}.\hskip 1em plus 0.5em minus
  0.4em\relax Athena scientific Belmont, 1999.

\bibitem{dhingra2016proximal}
N.~K. Dhingra, S.~Z. Khong, and M.~R. Jovanovi{\'c}, ``The proximal augmented
  lagrangian method for nonsmooth composite optimization,'' \emph{arXiv
  preprint arXiv:1610.04514}, 2016.

\bibitem{logit_regression}
\BIBentryALTinterwordspacing
(2012) Logistic regression. [Online]. Available:
  \url{http://www.stat.cmu.edu/~cshalizi/uADA/12/lectures/ch12.pdf}
\BIBentrySTDinterwordspacing

\end{thebibliography}

\appendix

\subsection{Proof of Lemma~\ref{lem:qp}}\label{appendix:qp}
{\color{black}Recall the definition of $G(\vv{z})$ in \eqref{eq:pdgd_linear}, 
{\small$$G(\vv{z}) =  \left[ \begin{array}{cc}
- B(\vv{x})  &  - A^T  \\
\eta A & 0
\end{array}  \right] $$}It can be seen that $G(\vv{z})$ depends on $\vv{z}$ through $B(\vv{x})$, and $B(\vv{x})$ satisfies $\mu I\preceq B(\vv{x})\preceq \ell I$.
In the remaining of this section, we will drop the dependence of $G(\vv{z})$ and $B(\vv{x})$ on $\vv{z}$ and $\vv{x}$, and prove $G^TP + PG\preceq -\tau P$, for \emph{any} symmetric $B$ satisfying $\mu I\preceq B\preceq \ell I$. Let $Q = - G^TP - PG$, then,}
{\small	\begin{align*}
	Q -\tau P
	&= \left[\begin{array}{cc}
	2\eta c B  -2 \eta^2 A^TA  - \eta^2 \kappa_1 I & \eta B  A^T - \frac{\eta^2 \kappa_1}{c} A^T\\
	\eta A B  - \frac{\eta^2 \kappa_1}{c}  A & 2\eta A A^T - \eta \kappa_1 I
	\end{array}\right] \\
	&\succeq \left[\begin{array}{cc}
	2\eta c B  -2 \eta^2  \kappa_2 I  - \eta^2 \kappa_1 I & \eta B A^T - \frac{\eta^2 \kappa_1}{c} A^T\\
	\eta A B - \frac{\eta^2 \kappa_1}{c}  A & \eta A A^T 
	\end{array}\right] 
	\end{align*}}where we have used $AA^T \succeq \kappa_1 I$, $A^T A \preceq \kappa_2 I$. We will next use the Schur complement argument. Consider
{\small	\begin{align*}
	&( \eta B A^T - \frac{\eta^2 \kappa_1}{c} A^T ) (\eta AA^T)^{-1} (	\eta A B - \frac{\eta^2 \kappa_1}{c}  A) \\
	&=  \eta B  A^T (AA^T)^{-1}AB - \frac{\eta^2 \kappa_1}{c}(B A^T(AA^T)^{-1}A \\
	&\qquad+ A^T (AA^T)^{-1} AB ) + \frac{\eta^3 \kappa_1^2}{c^2} A^T (AA^T)^{-1} A\\
	&\preceq \eta\ell  B + \frac{\eta^2 \kappa_1}{c} 2 \ell  I +  \frac{\eta^3 \kappa_1^2}{c^2} I 
	\end{align*}}where we have used $A^T (AA^T)^{-1} A\preceq I$. Recall that $c=4\max(\ell,\frac{\eta \kappa_2}{\mu})$, $\kappa_1\leq \kappa_2$ and $\mu\leq l$. Then we have, i) $\frac{3}{4} \eta c B \succeq \frac{3}{4} \eta c \mu I \succeq   3  \eta^2 \kappa_2 I \succeq 2 \eta^2 \kappa_2 I + \eta^2 \kappa_1 I$, ii) $\frac{1}{4} \eta c B  \succeq \eta \ell   B $, iii)
	 $ \frac{1}{2} \eta c B  \succeq  \frac{1}{2} \eta c \mu I \succeq  \frac{\eta^2}{c} 2 \ell \kappa_1 I$, and iv)$ \frac{1}{2} \eta c B  \succeq \frac{1}{2} \eta c \mu I \succeq   \frac{\eta^3 \kappa_1^2}{c^2} I$. Summing them up, we have
{\small	\begin{align*}
	2 \eta c B  &\succeq 2 \eta^2 \kappa_2 I + \eta^2 \kappa_1 I + \eta \ell  B + \frac{\eta^2}{c} 2 \ell \kappa_1 I   +  \frac{\eta^3 \kappa_1^2}{c^2} I \\
	&\succeq 2 \eta^2 \kappa_2 I + \eta^2 \kappa_1 I \\
	&+ ( \eta B  A^T - \frac{\eta^2 \kappa_1}{c} A^T ) (\eta AA^T)^{-1} (	\eta A B  - \frac{\eta^2 \kappa_1}{c}  A)
	\end{align*}}As a result, we have $Q  - \tau P\succeq 0$. 
\vspace{-6pt}
\subsection{Proof of Lemma~\ref{lem:ineq:qp}}\label{appendix:ineq:qp}
{\color{black}Recall the definition of $G(\vv{z})$ in \eqref{eq:ineq:linear}, 
{\small$$G(\vv{z}) = \left[ \begin{array}{cc}
- B(\vv{x}) - \rho A^T\Gamma(\vv{z}) A &  - A^T \Gamma(\vv{z}) \\
\eta \Gamma(\vv{z}) A & \frac{ \eta}{\rho}(\Gamma(\vv{z}) - I)
\end{array}  \right] .$$}It can be seen that $G(\vv{z})$ depends on the state $\vv{z}$ through $B(\vv{x})$, $\Gamma(\vv{z})$. Note for any $\vv{z}\in\R^{n+m}$, $B(\vv{x})$ is a symmetric matrix satisfying $\mu I\preceq B(\vv{x})\preceq \ell I$, and $\Gamma(\vv{z})$ is a diagonal matrix with each entry in $[0,1]$. In the following, to simplify notation, we will drop the dependence of $G$, $B$, $\Gamma$ on $\vv{z}$, and prove $G^TP + PG\preceq -\tau P$ for any symmetric $B$ satisfing $\mu I \preceq B\preceq \ell I$ and for any diagonal $\Gamma$ with each entry bounded in $[0,1]$. Let $Q = -G^TP - PG$, and }
{\small$$\tilde{Q} = Q - \tau P =\left[ \begin{array}{cc}
Q_1 & Q_3\\
Q_3^T& Q_2
\end{array}\right] $$}After straightforward calculations, we have
{\small\begin{align*}
Q_1 &= 2\eta c B+ 2(\eta c\rho - \eta^2) A^T\Gamma A - \frac{\eta^2\kappa_1}{2} I \\ 
Q_2 & = \eta (\Gamma A A^T + AA^T\Gamma) + \frac{2\eta c}{\rho} (I - \Gamma ) - \frac{\eta\kappa_1}{2} I \\
Q_3 & = \eta B A^T + \eta \rho A^T\Gamma A A^T + \frac{\eta^2}{\rho} A^T(I - \Gamma) -  \frac{\eta^2\kappa_1}{2c} A^T \\
&= \eta(\underbrace{ B  + \overbrace{ \rho A^T\Gamma A  - \frac{\eta\kappa_1}{2c}  I }^{Q_5}}_{Q_4})A^T + \eta \cdot \underbrace{\frac{\eta}{\rho} A^T(I - \Gamma) }_{Q_6}
\end{align*}}{\color{black}Using the Schur complement argument, to prove $\tilde{Q} \succeq 0$ it suffices to prove $Q_2\succ 0$, $Q_1 - Q_3 Q_2^{-1} Q_3^T \succeq 0$. To show this, we will first lower bound $Q_2$, then upper bound $Q_3 Q_2^{-1} Q_3^T$, next lower bound $Q_1 $ and finally show $Q_1 - Q_3 Q_2^{-1} Q_3^T \succeq 0$.}

\noindent\textbf{\textcolor{black}{Lower bounding $Q_2$.}} We will use the following lemma, whose proof is deferred to Appendix-\ref{subsec:ineq:gamma}.
\begin{lemma}\label{lem:ineq:gamma}
	If $c\geq \kappa_2\rho$, as long as $\Gamma$ is a diagonal matrix with each entry bounded in $[0,1]$, we have $\eta (\Gamma A A^T + AA^T\Gamma) + \frac{2\eta c}{\rho} (I - \Gamma ) \succeq \frac{3}{2} \eta AA^T  $.
\end{lemma}
Using Lemma~\ref{lem:ineq:gamma}, we have $Q_2 \succeq \frac{3}{2}\eta AA^T - \frac{\eta\kappa_1}{2} I \succeq \eta AA^T $. 
{\color{black}\textbf{Upper bounding $Q_3 Q_2^{-1} Q_3^T$.}} Using the lower bound on $Q_2$, 
{\small\begin{align}
& Q_3 Q_2^{-1} Q_3^T       \preceq \frac{1}{\eta} Q_3 (AA^T)^{-1} Q_3^T \nonumber\\
&=  \eta Q_4 A^T(AA^T)^{-1} A Q_4^T  + \eta Q_4 A^T (AA^T)^{-1} Q_6^T\nonumber\\
&\qquad + \eta Q_6 (AA^T)^{-1} A Q_4^T+ \eta Q_6(AA^T)^{-1} Q_6^T \nonumber \\
& \preceq \eta Q_4 Q_4^T + 2\eta \Vert  Q_4 A^T (AA^T)^{-1} Q_6^T\Vert I \nonumber\\
&\qquad +\eta \Vert Q_6(AA^T)^{-1} Q_6^T\Vert  I  \label{eq:ineq:Q323}
\end{align}}where in the last inequality we have used $A^T(AA^T)^{-1} A \preceq I$. To further bound $Q_3 Q_2^{-1} Q_3^T$, we use the following,
{\small\begin{align}
2\eta \Vert  Q_4 A^T (AA^T)^{-1} Q_6^T\Vert  &\leq 2\eta \Vert Q_4 \Vert \| A^T\Vert \|(AA^T)^{-1}\| \|Q_6\| \nonumber \\
& \leq 2\frac{\eta^2}{\rho} (\ell +  \rho\kappa_2+  \frac{\eta  \kappa_1}{2c } ) \frac{  \kappa_2}{\kappa_1}    \coloneqq h_1(c)  \label{eq:ineq:h1}
\end{align}}where we intentionally write the last quantity as a function $h_1(c)$ depending on $c$ for reasons to be clear later. We further bound 
{\small\begin{align}
\eta \Vert Q_6(AA^T)^{-1} Q_6^T\Vert  & \leq \eta \| Q_6\|^2 \| (AA^T)^{-1}\| \leq \frac{\eta^3}{\rho^2} \frac{\kappa_2}{\kappa_1} \coloneqq h_2(c)  \label{eq:ineq:h2}
\end{align}}Then, we bound $\eta Q_4 Q_4^T$ as 
{\small\begin{align}
\eta Q_4 Q_4^T & = \eta (B + Q_5) ^2  \preceq \eta B^2 + 2\eta \| B \| \|Q_5\| I +  \eta \|Q_5\|^2 I \nonumber\\
&\preceq \eta \ell B  + (\underbrace{ 2\eta \ell ( \rho\kappa_2 + \frac{\eta \kappa_1}{2c} ) + \eta ( \rho\kappa_2 + \frac{\eta \kappa_1}{2c} )^2}_{:=h_3(c)} )I \label{eq:ineq:h3}
\vspace{-6pt}
\end{align}}Combining (\ref{eq:ineq:Q323}) (\ref{eq:ineq:h1}) (\ref{eq:ineq:h2}) (\ref{eq:ineq:h3}), we have $Q_3 Q_2^{-1} Q_3^T \preceq \eta \ell B+ (h_1(c) + h_2(c)+h_3(c))I$. 

\noindent\textbf{Lower bounding $Q_1$.} It is easy to obtain,
\vspace{-6pt}
\small{\begin{align}
Q_1 \succeq 2\eta c B  - ( \underbrace{2\eta^2 \kappa_2 + \frac{\eta^2\kappa_1}{2 }}_{:=h_4(c) } )I  \label{eq:ineq:h4}
\end{align}}

\noindent\textbf{Proving $Q_1 - Q_3 Q_2^{-1} Q_3^T \succeq 0$.}
Combining the above bounds on $Q_1$ and $Q_3 Q_2^{-1} Q_3^T$, we have,
\begin{align*}
& Q_1 - Q_3 Q_2^{-1} Q_3^T \\
& \succeq (2\eta c - \eta \ell) B  - ( h_1(c) + h_2(c)+ h_3(c) + h_4(c) ) I  \\
& \succeq \big[2\eta \mu c -( \eta \ell \mu  +  h_1(c) + h_2(c)+ h_3(c) + h_4(c) )\big] I  
\end{align*}
Since $h_1(c) + h_2(c) + h_3(c) + h_4(c) $ is a strictly decreasing function in $c$, we have $Q_1 - Q_3 Q_2^{-1} Q_3^T \succeq 0$ as long as $c$ is sufficiently large.
 We can verify that our selection of $c$ is large enough s.t. $Q_1 - Q_3 Q_2^{-1} Q_3^T \succeq 0$. Due to space limit, we omit the details. Therefore, $Q  \succeq \tau P$. 

\vspace{-6pt}
  \subsection{Proof of Lemma~\ref{lem:ineq:gamma}} \label{subsec:ineq:gamma}
  Recall that $\Gamma = \text{diag}(\gamma_1,\ldots,\gamma_m)$ with each $\gamma_i \in[0,1]$. The matrix of interest is $M(\gamma_1,\ldots,\gamma_m) = \eta (\Gamma A A^T + AA^T\Gamma) + \frac{2\eta c}{\rho} (I - \Gamma ) $. 
  It is easy to check $M(\gamma_1,\ldots,\gamma_m) $ is a convex combination of $2^m$ matrices $\{M(b_1,\ldots,b_m): b_i = 0$ or $1\}$.  Therefore, to prove the lower bound for $M(\gamma_1,\ldots,\gamma_m)$, without loss of generality, we only have to prove that for $k=0,\ldots,m${\small
  \begin{align}
  M_k = M(\underbrace{1,1,\ldots,1}_{k \text{ entries}},0,\ldots,0)\succeq \frac{3}{2} \eta  AA^T \label{eq:ineq:gamma:m_lb}
  \end{align}}Notice that $M_0 = \frac{2\eta c}{\rho} I \succeq \frac{3}{2} \eta AA^T  $, $M_m = 2 \eta AA^T$, so (\ref{eq:ineq:gamma:m_lb}) is true for $k=0$ and $m$. Now assume $0<k<m$. We write $AA^T$ in block matrix form
  {\small\begin{align}
  AA^T = \left [ \begin{array}{cc}
  \Lambda_1 & \Lambda_3\\
  \Lambda_3^T & \Lambda_2
  \end{array} \right]
  \end{align}}with $\Lambda_1 \in \R^{k\times k}, \Lambda_2\in\R^{(m-k)\times (m-k)}, \Lambda_3\in\R^{(k)\times (m-k)} $. Then, we can write $M_k$ as 
  {\small\begin{align}
  M_k = \left [ \begin{array}{cc}
  2 \eta \Lambda_1 & \eta\Lambda_3\\
  \eta \Lambda_3^T &  2\frac{\eta c}{\rho} I 
  \end{array} \right]\succeq \left [ \begin{array}{cc}
  2 \eta \Lambda_1 & \eta\Lambda_3\\
  \eta \Lambda_3^T &  2 \eta \Lambda_2 
  \end{array} \right]\succeq \frac{3}{2} \eta  AA^T \nonumber
  \end{align}}where we have used the fact that $\Lambda_2 \preceq \Vert \Lambda_2 \Vert I \preceq \kappa_2 I \preceq  \frac{c}{\rho} I $ (using $c \geq \rho \kappa_2$). 
\ifthenelse{\boolean{isfullversion}}{

\subsection{Proof of Lemma~\ref{lem:B}}\label{appendix:proof_B}
\begin{proof} Define $\vv{g}(t) = \nabla f(\vv{x}^*+t(\vv{x}-\vv{x}^*))$ where $t\in [0,1]$. Then, for each $1\leq i \leq n$, we have 
	\begin{align*}
	[\frac{d}{dt} \vv{g}(t)]_i = \sum_{j=1}^n \frac{\partial^2 f(\vv{x}^*+t(\vv{x}-\vv{x}^*))}{\partial x_i \partial x_j} (x_j - x_j^*). 
	\end{align*}
	Therefore, $\frac{d}{dt}  \vv{g}(t) = \nabla^2   f(\vv{x}^*+t(\vv{x}-\vv{x}^*)) (\vv{x}-\vv{x}^*) $, and 
	\begin{align*}
	&\nabla f(\vv{x}) - \nabla f(\vv{x}^* ) = \vv{g}(1) - \vv{g}(0) = \int_0^1 \frac{d}{dt}  \vv{g}(t) dt\\
	&= \underbrace{\Big( \int_0^1 \nabla^2   f(\vv{x}^*+t(\vv{x}-\vv{x}^*))  dt \Big)}_{:= B({\vv{x}})} (\vv{x}-\vv{x}').
	\end{align*}
	Since for any $t\in[0,1]$ we have $\mu I\preceq \nabla^2   f(\vv{x}^*+t(\vv{x}-\vv{x}^*)) \preceq \ell I $, therefore $\mu I\preceq B({\vv{x}})\preceq \ell I$. 
\end{proof}

\subsection{Proof of Proposition~\ref{thm:ineq:fixedpoint}}\label{appendix:ineq:fixedpoint}
The crucial observation is that the fixed point equation $\dot{\vv{x}} = 0, \dot{\vv{\lambda}} = 0$ of Aug-PDGD (\ref{eq:ineq:pdgd_p}) (\ref{eq:ineq:pdgd_d}) is the same as the KKT condition of problem~\eqref{eq:ineq:opt_problem}. To see this, notice that $\dot{\vv{\lambda}} = 0$ can be written as  $\forall j,  \max(\rho(\vv{a}_j^T\vv{x} - b_j) + \lambda_j,0) = \lambda_j$, 
which is equivalent to
\begin{align}
\lambda_j&\geq 0 \label{eq:ineq:fixedpoint:dual_fea} \\
 \vv{a}_j^T\vv{x} - b_j&\leq 0 \label{eq:ineq:fixedpoint:primal_fea} \\
 (\vv{a}_j^T\vv{x} - b_j)\lambda_j &= 0 \label{eq:ineq:fixedpoint:com_sla}
\end{align}	
The equation $\dot{\vv{x}} = 0$ can be rewritten as
\begin{align}
0&= \nabla f(\vv{x}) + \sum_{j=1}^m \max(\rho(\vv{a}_j^T\vv{x} - b_j) + \lambda_j,0) \vv{a}_j \nonumber\\
&=\nabla f(\vv{x}) + \sum_{j=1}^m \lambda_j \vv{a}_j  = \nabla f(\vv{x}) + A^T\vv{\lambda}\label{eq:ineq:fixedpoint:stationary}
\end{align}
where the second equality is due to $\max(\rho(\vv{a}_j^T\vv{x} - b_j) + \lambda_j,0) = \lambda_j$ (implied by $\dot{\vv{\lambda}} = 0$). Equation \eqref{eq:ineq:fixedpoint:dual_fea} \eqref{eq:ineq:fixedpoint:primal_fea} \eqref{eq:ineq:fixedpoint:com_sla} \eqref{eq:ineq:fixedpoint:stationary} are precisely the KKT condition of problem~\eqref{eq:ineq:opt_problem}.

 \subsection{Proof of Lemma~\ref{lem:discretization}}\label{appendix:discretization}
%

	We will use notation $\Vert \vv{y}\Vert_P$ to denote norm $\sqrt{\vv{y}^T P\vv{y} }$. Since $F$ is $\nu$-Lipschitz continuous in the Euclidean norm, it is also $\nu'$-Lipschitz continuous in the $\Vert \cdot\Vert_P$ norm, where $\nu' = \nu \sqrt{\kappa_P}$. We fix $k$, and consider the dynamics $\dot{\vv{z}} = F (\vv{z}) $ that starts at $\vv{z}(0) = \vv{y}(k)$. The property of the Lyapunov function $V$ implies that $\Vert \vv{z}(t) - \vv{z}^* \Vert_P \leq e^{-\frac{\tau t}{2}} \Vert \vv{z}(0) - \vv{z}^*\Vert_P$. Then we have for any $t>0$,
\begin{align}
\Vert \vv{z}(t) - \vv{z}(0)\Vert_P &= \Vert \int_{\xi =0}^t F(\vv{z}(\xi)) - F(\vv{z}^* ) d\xi\Vert_P \nonumber\\
&\leq \nu' \int_{\xi =0}^t \Vert \vv{z}(\xi) - \vv{z}^*\Vert_P d\xi\nonumber\\
&\leq\nu'  \int_{\xi =0}^t \Vert \vv{z}(0) - \vv{z}^*\Vert_P d\xi \nonumber \\
& = \nu' t \Vert \vv{z}(0) - \vv{z}^*\Vert_P \label{eq:discrete:ztz0}
\end{align}
Next, we bound
\begin{align*}
\Vert \vv{z}(\delta) - \vv{y}(k+1)\Vert_P&= \Vert \vv{z}(\delta) - \vv{z}(0)- (\vv{y}(k+1) -\vv{y}(k))\Vert_P \\
&=\Vert \int_{t=0}^\delta F(\vv{z}(t)) - F(\vv{z}(0))  dt\Vert_P \\
&\leq  \nu' \int_{t=0}^\delta\Vert  \vv{z}(t) -  \vv{z}(0) \Vert_P  dt \\
&\leq \nu'^2 \int_{t=0}^\delta t \Vert  \vv{z}(0) -  \vv{z}^* \Vert_P  dt \\
&= \frac{\nu'^2 \delta^2}{2}\| \vv{y}(k)-  \vv{z}^*\|_P
\end{align*}
where in the last inequality we have used (\ref{eq:discrete:ztz0}). At last, we have 
{\small\begin{align*}
\Vert \vv{y}(k+1) - \vv{z}^* \Vert_P& \leq \Vert \vv{y}(k+1) - \vv{z}(\delta)\Vert_P + \Vert\vv{z}(\delta) - \vv{z}^*\Vert_P\\
&\leq \frac{\nu'^2 \delta^2}{2}\| \vv{y}(k)-  \vv{z}^*\|_P + e^{-\frac{\tau \delta}{2}}\Vert \vv{z}(0) - \vv{z}^*\Vert_P \\
&= (e^{-\frac{\tau \delta}{2}} + \frac{\nu'^2 \delta^2}{2} )\| \vv{y}(k)-  \vv{z}^*\|_P \\
&\leq (e^{-\frac{\tau \delta}{2}} + \frac{\nu'^2 \delta^2}{2} )^{k+1}\| \vv{y}(0)-  \vv{z}^*\|_P.
\end{align*}}The final claim of the lemma follows from the fact that quantity $e^{-\frac{\tau \delta}{2}} + \frac{\nu'^2 \delta^2}{2}$ as a function of $\delta$, equals $1$ when $\delta = 0$, and has negative derivative w.r.t. $\delta$ at $\delta = 0$. 
\subsection{Two Sided Inequality Constraints}\label{appendix:twosided}
In this section, we study an extended version of the inequality constrained problem in Section~\ref{subsec:ineq:problem_formulation}, where the one-sided constraints are replaced with two sided constraints. The notations in this section is independent from the rest of the paper, though we use the same letter as the rest of the paper to represent the same type of quantity. 

 We consider the following problem with two-sided inequality constraints,
\begin{align}
	\min_{\vv{x}\in \R^n} \quad & f(\vv{x}) \label{eq:twosided:opt_problem} \\
	\text{s.t.} \quad \bar{\vv{b}} &  \leq A\vv{x} \leq \underline{\vv{b}} \nonumber
\end{align}
where $A\in\R^{m\times n}$, and $\bar{\vv{b}} = [\bar{b}_1,\ldots,\bar{b}_m]^T$, $\underline{\vv{b}} = [\underline{b}_1,\ldots,\underline{b}_m]^T\in \R^m$. We still use Assumption~\ref{assump:f}, \ref{assump:A}, that is, $f$ is $\mu$ strongly convex and $\ell$-smooth, $A$ is full row rank and  $\kappa_1 I \preceq  AA^T\preceq \kappa_2 I $. We further assume without loss of generality that $\underline{b}_j < \bar{b}_j$, for $j=1,\ldots,m$. 

The algorithms, theorems and analysis in this section basically follow the same line as the one-sided constraint case in Section~\ref{subsec:ineq:problem_formulation} and \ref{subsec:ineq}. The major difference is that we use a slightly different penalty function when formalizing the Augmented Lagrangian. The Augmented Lagrangian is given by,
\begin{align}
L(\vv{x},\vv{\lambda}) = f(\vv{x}) + \sum_{j=1}^m H_\rho(\vv{a}_j^T\vv{x}, \underline{b}_j,\bar{b}_j,\lambda_j) \label{eq:twosided:Lagrangian}
\end{align}
where $\rho>0$ is a free parameter, and the penalty function $H_\rho$ is given as follows,
\begin{align*}
H_\rho(\vv{a}_j^T \vv{x},  \underline{b}_j,\bar{b}_j , \lambda_j) 
&= \left\{ \begin{array}{ll}
(\vv{a}_j^T \vv{x} - \underline{b}_j ) \lambda_j     + \frac{\rho}{2}(\vv{a}_j^T \vv{x} -\underline{b}_j  )^2  & \text{if } \rho(\vv{a}_j^T \vv{x}) + \lambda_j <  \rho \underline{b}_j\\
(\vv{a}_j^T \vv{x} - \bar{b}_j )\lambda_j    + \frac{\rho}{2} (\vv{a}_j^T\vv{x} - \bar{b}_j)^2 & \text{if } \rho(\vv{a}_j^T\vv{x} ) + \lambda_j >\rho \bar{b}_j  \\
-\frac{1}{2} \frac{\lambda_j^2}{\rho}& \text{otherwise. } \end{array}\right.
\end{align*}
And the derivative of the penalty function w.r.t. $\vv{x}$ is given by,
\begin{align*}
\nabla_{\vv{x}} H_\rho(\vv{a}_j^T \vv{x},  \underline{b}_j,\bar{b}_j , \lambda_j)  &= \left\{ \begin{array}{ll}
    (  \rho (\vv{a}_j^T \vv{x} -\underline{b}_j  ) +\lambda_j)\vv{a}_j & \text{if } \rho(\vv{a}_j^T \vv{x}) + \lambda_j < \rho  \underline{b}_j\\
      (\rho (\vv{a}_j^T\vv{x} - \bar{b}_j)  + \lambda_j ) \vv{a}_j& \text{if } \rho(\vv{a}_j^T\vv{x} ) + \lambda_j > \rho \bar{b}_j \\
\vv{0} & \text{otherwise. } \end{array}\right. \\
 &= S_{\rho\underline{b}_j}^{\rho \bar{b}_j }(\rho (\vv{a}_j^T \vv{x} ) +\lambda_j  ) \vv{a}_j
\end{align*}
where for $\underline{\alpha}< \bar{\alpha}$, $S_{\underline{\alpha}}^{\bar{\alpha}}: \R\rightarrow \R$ is the soft thresholding function, defined as
$$S_{\underline{\alpha}}^{\bar{\alpha}} (y) =  \max [ \min( y - \underline{\alpha},0), y - \bar{\alpha} ] . $$
Similarly, we can calculate the derivative w.r.t. $\vv{\lambda}$,
\begin{align*}
\nabla_{\vv{\lambda}} H_\rho(\vv{a}_j^T \vv{x},  \underline{b}_j,\bar{b}_j , \lambda_j) 
&= \left\{ \begin{array}{ll}
   (\vv{a}_j^T \vv{x} -\underline{b}_j   )\vv{e}_j & \text{if } \rho(\vv{a}_j^T \vv{x}) + \lambda_j < \rho  \underline{b}_j\\
  (\vv{a}_j^T\vv{x} - \bar{b}_j)     \vv{e}_j& \text{if } \rho(\vv{a}_j^T\vv{x} ) + \lambda_j > \rho \bar{b}_j \\
-\frac{1}{\rho} \lambda_j \vv{e}_j & \text{otherwise. } \end{array}\right. \\
&=\frac{ S_{\rho\underline{b}_j}^{\rho \bar{b}_j }(\rho (\vv{a}_j^T \vv{x} ) +\lambda_j  ) - \lambda_j }{\rho }\vv{e}_j.
\end{align*}
We can then write down the primal dual gradient dynamics for \eqref{eq:twosided:Lagrangian} as follows,
{\begin{subeqnarray} \label{eq:twosided:pdgd}
\dot{\vv{x}} & = &- \nabla_{\vv{x}} L(\vv{x},\vv{\lambda})  \nonumber \\
&=& - \nabla f(\vv{x}) - \sum_{j=1}^m S_{\rho\underline{b}_j}^{\rho \bar{b}_j }(\rho (\vv{a}_j^T \vv{x} ) +\lambda_j ) \vv{a}_j   \slabel{eq:twosided:pdgd_p}\\
\dot{\vv{\lambda}} &=& \eta \nabla_{\vv{\lambda}} L(\vv{x},\vv{\lambda})  \nonumber \\
&=& \eta \sum_{j=1}^m \frac{ S_{\rho\underline{b}_j}^{\rho \bar{b}_j }(\rho (\vv{a}_j^T \vv{x} ) +\lambda_j ) - \lambda_j }{\rho }\vv{e}_j \slabel{eq:twosided:pdgd_d}
\end{subeqnarray}}We name the dynamics \eqref{eq:twosided:pdgd} as Aug-PDGD-TS (Augmented Primal Dual Gradient Dynamics for Two Sided inequality constraints). Again, we have the following the lemma regarding the fixed point of the Aug-PDGD-TS.
\begin{theorem}
	The Aug-PDGD-TS (\ref{eq:twosided:pdgd}) has a unique fixed point, $(\vv{x}^*,\vv{\lambda}^*)$. Moreover, let $\bar{\vv{\lambda}} $ be a vector such that the $j$'th entry being $\max(\lambda_i^*,0)$, and $\underline{\vv{\lambda}}$ be a vector such that the $j$'th entry being $-\min(\lambda_i^*,0)$. Then, $(\vv{x}^*, \bar{\vv{\lambda}},\underline{\vv{\lambda}})$ satisfies the KKT-condition of problem~\eqref{eq:twosided:opt_problem}.
\end{theorem}
\begin{proof}
	The crucial observation is that, equation $\dot{\vv{\lambda}} = 0$, can be written as $\forall j, \lambda_j = S_{\rho\underline{b}_j}^{\rho \bar{b}_j }(\rho (\vv{a}_j^T \vv{x} ) + \lambda_j)$, and is equivalent to, 
	\begin{align}
		\vv{a}_j^T \vv{x}&\leq \bar{b}_j \label{eq:twosided:fixedpoint:pf_1}\\
		\vv{a}_j^T \vv{x} &\geq \underline{b}_j \label{eq:twosided:fixedpoint:pf_2}\\
		(\vv{a}_j^T \vv{x} - \underline{b}_j )\min( \lambda_j, 0) &=0 \label{eq:twosided:fixedpoint:cs_1} \\
		(\vv{a}_j^T \vv{x} - \bar{b}_j )\max(\lambda_j,0)& = 0. \label{eq:twosided:fixedpoint:cs_2}
	\end{align}
	Further, the equation $\dot{\vv{x}} = 0$ can be rewritten as
	\begin{align}
	0 &= \nabla f(\vv{x}) +  \sum_{j=1}^m S_{\rho\underline{b}_j}^{\rho \bar{b}_j }(\rho (\vv{a}_j^T \vv{x} ) +\lambda_j ) \vv{a}_j \nonumber \\
	&= \nabla f(\vv{x}) +  \sum_{j=1}^m \lambda_j \vv{a}_j  \nonumber \\
	&= \nabla f(\vv{x}) +  \sum_{j=1}^m \max(\lambda_j,0) \vv{a}_j + \sum_{j=1}^m (- \min(\lambda_j,0) ) (-\vv{a}_j) \label{eq:twosided:fixedpoint:stationary}
	\end{align}
 If we replace $\max(\lambda_j,0)$ with $\bar{\lambda}_j$, and replace $-\min(\lambda_j,0)$ with $\underline{\lambda}_j$, then \eqref{eq:twosided:fixedpoint:pf_1} \eqref{eq:twosided:fixedpoint:pf_2} \eqref{eq:twosided:fixedpoint:cs_1} \eqref{eq:twosided:fixedpoint:cs_2} \eqref{eq:twosided:fixedpoint:stationary}, together with $\bar{\lambda}_j\geq 0, \underline{\lambda}_j\geq 0$, are precisely the KKT conditions of problem~\eqref{eq:twosided:opt_problem}. Finally, notice that the mapping $\R\rightarrow [0,\infty)\times [0,\infty), \lambda\mapsto (\max(\lambda,0),-\min(\lambda,0) )$ is a bijection. So there is a one-to-one correspondence between the fixed point of the Aug-PDGD-TS (\ref{eq:twosided:pdgd}) and the solution of the KKT condition of \eqref{eq:twosided:opt_problem}. 
\end{proof}
Then, we state the following exponential stability result, whose proof is deferred to Appendix-\ref{appendix:twosided:proof_stability}. 
\begin{theorem}\label{thm:twosided:stablility}
	Under Assumption~\ref{assump:f} and \ref{assump:A}, the Aug-PDGD-TS (\ref{eq:twosided:pdgd}) is exponential stable in the sense that for any $\eta>0$, $\rho>0$, there exists constant {\small$$\tau_{ineq,ts} = \frac{\eta\kappa_1^2}{40 \ell \kappa_2 [ \max(\frac{\rho \kappa_2}{\mu}, \frac{\ell}{\mu})]^2 [\max(\frac{\eta}{\ell \rho}, \frac{\ell}{\mu})]^2 }$$}and constants $C_3,C_4>0$ which depend on $\eta,\rho,\kappa_1,\kappa_2,\mu,\ell$, $\Vert \vv{x}(0) - \vv{x}^*\Vert$, $\Vert \vv{\lambda}(0) - \vv{\lambda}^* \|$, s.t. $\Vert \vv{x}(t) - \vv{x}^*\Vert\leq C_3 e^{-\tau_{ineq,ts} t}$, $\Vert \vv{\lambda}(t) - \vv{\lambda}^*\Vert\leq C_4e^{-\tau_{ineq,ts} t}$.
\end{theorem}

\subsection{Proof of Theorem~\ref{thm:twosided:stablility}} \label{appendix:twosided:proof_stability}
Our proof strategy is similar as the one-sided constraint case in Section~\ref{subsec:ineq}. Let $\vv{z} = [  \vv{x}^T,\vv{\lambda} ^T ]^T$ and define $\vv{z}^*$ similarly. We will first write the Aug-PDGD-TS in a ``linear'' form, $\frac{d}{dt} {\vv{z}} = G(\vv{z}) (\vv{z} - \vv{z}^*)$ in (\ref{eq:twosided:linear}), where the state transition matrix $G(\vv{z})$ depends on the state $\vv{z}$. To do this, we rely on the following Lemma.
\begin{lemma}\label{lem:twosided:gamma}
	We have that, for any $j$ and state $\vv{z}$, there exists $\gamma_j(\vv{z})\in[0,1]$ that depends on $\vv{z} $, s.t.
	\begin{align*}
	 \nabla_{\vv{x}} H_\rho(\vv{a}_j^T \vv{x},  \underline{b}_j,\bar{b}_j , \lambda_j) - \nabla_{\vv{x}} H_\rho(\vv{a}_j^T \vv{x}^*,  \underline{b}_j,\bar{b}_j , \lambda_j^*) 
	&= \gamma_j(\vv{z}) \rho \vv{a}_j^T (\vv{x} - \vv{x}^* ) \vv{a}_j + \gamma_j (\vv{z})(\lambda_j - \lambda_j^*) \vv{a}_j \\
	\nabla_{\vv{\lambda}} H_\rho(\vv{a}_j^T \vv{x},  \underline{b}_j,\bar{b}_j , \lambda_j) - \nabla_{\vv{\lambda}} H_\rho(\vv{a}_j^T \vv{x}^*,  \underline{b}_j,\bar{b}_j , \lambda_j^*) 
	&= \gamma_j(\vv{z}) \vv{a}_j^T (\vv{x} - \vv{x}^*)  \vv{e}_j + \frac{1}{\rho}(\gamma_j(\vv{z}) - 1) (\lambda_j - \lambda_j^*) \vv{e}_j
	\end{align*} 
\end{lemma}
\begin{proof}
	Observe for any $\underline{\alpha}< \bar{\alpha}$, $S_{\underline{\alpha}}^{\bar{\alpha}}(\cdot)$ is an increasing function and $1$-Lipschitz continuous. Therefore, for any $y,y^*\in\R$, there exists some $\gamma\in[0,1]$ which depends on $y$, $y^*$, s.t. 
	$$S_{\underline{\alpha}}^{\bar{\alpha}} (y) - S_{\underline{\alpha}}^{\bar{\alpha}} (y^*) = \gamma (y - y^*) .$$
	The statement of the lemma directly follows from the above result. 
\end{proof}
Also using Lemma~\ref{lem:B}, we have $\nabla f(\vv{x}) - \nabla f(\vv{x}^*) = B (\vv{x}) (\vv{x} - \vv{x}^*)$ where $\mu I\preceq B(\vv{x})\preceq \ell I$. For any $\vv{z}$, we also define $\Gamma(\vv{z}) = \text{diag}(\gamma_1(\vv{z}),\ldots,\gamma_m(\vv{z}))$ where $\gamma_j(\vv{z})$ is from Lemma~\ref{lem:twosided:gamma}. Then, we rewrite (\ref{eq:twosided:pdgd_p}) as
\begin{align*}
\dot{\vv{x}} &= - (\nabla_{\vv{x} } L(\vv{x},\vv{\lambda}) - \nabla_{\vv{x} } L(\vv{x}^*,\vv{\lambda}^*)  )\\
&= - (\nabla f(\vv{x} ) - \nabla f(\vv{x}^*))    - \sum_{j=1}^m ( \nabla_{\vv{x}} H_\rho(\vv{a}_j^T \vv{x},  \underline{b}_j,\bar{b}_j , \lambda_j) - \nabla_{\vv{x}} H_\rho(\vv{a}_j^T \vv{x}^*,  \underline{b}_j,\bar{b}_j , \lambda_j^*) ) \\
&= - B({\vv{x}}) (\vv{x} - \vv{x}^*)  - \sum_{j=1}^m \Big[ \gamma_j(\vv{z}) \rho \vv{a}_j^T (\vv{x} - \vv{x}^* ) \vv{a}_j + \gamma_j (\vv{z})(\lambda_j - \lambda_j^*) \vv{a}_j  \Big] \\
&= - B({\vv{x}}) (\vv{x} - \vv{x}^*)  -\rho \Big[  \sum_{j=1}^m \gamma_j(\vv{z}) \vv{a}_j  \vv{a}_j^T\Big] (\vv{x} - \vv{x}^* ) -  A^T \Gamma(\vv{z}) (\vv{\lambda} - \vv{\lambda}^*) \\
& = - B({\vv{x}}) (\vv{x} - \vv{x}^*)  -\rho A^T\Gamma(\vv{z}) A (\vv{x} - \vv{x}^* ) -  A^T \Gamma(\vv{z}) (\vv{\lambda} - \vv{\lambda}^*) 
\end{align*}
Similarly, we rewrite (\ref{eq:twosided:pdgd_d}), as 
  \begin{align*}
  \dot{\vv{\lambda}} &= \eta (\nabla_{\vv{\lambda} } L(\vv{x},\vv{\lambda}) - \nabla_{\vv{\lambda} } L(\vv{x}^*,\vv{\lambda}^*)  )\\
  &= \eta \sum_{j=1}^m ( \nabla_{\vv{\lambda}} H_\rho(\vv{a}_j^T \vv{x},  \underline{b}_j,\bar{b}_j , \lambda_j) - \nabla_{\vv{\lambda}} H_\rho(\vv{a}_j^T \vv{x}^*,  \underline{b}_j,\bar{b}_j , \lambda_j^*) ) \\
  &= \eta \sum_{j=1}^m \Big[ \gamma_j (\vv{z})\vv{a}_j^T (\vv{x} - \vv{x}^*)  \vv{e}_j + \frac{1}{\rho}(\gamma_j (\vv{z})- 1) (\lambda_j - \lambda_j^*) \vv{e}_j \Big] \\
  &= \eta \Gamma(\vv{z}) A (\vv{x} - \vv{x}^*) + \frac{\eta}{\rho} (\Gamma(\vv{z}) - I) (\vv{\lambda} - \vv{\lambda}^*).
  \end{align*}
Hence, we can rewrite (\ref{eq:twosided:pdgd}) as
\begin{align}
\frac{d}{dt} {\vv{z}} = \left[\begin{array}{cc}
-B({\vv{x}}) - \rho A^T\Gamma(\vv{z}) A & - A^T\Gamma(\vv{z})\\
\eta\Gamma(\vv{z}) A &  \frac{\eta}{\rho}(\Gamma(\vv{z}) - I)
\end{array}\right] (\vv{z} - \vv{z}^*) \coloneqq G(\vv{z}) (\vv{z} - \vv{z}^*) \label{eq:twosided:linear}
\end{align}
Note that the above form of Aug-PDGD-TS appears the same as the Aug-PDGD in the one-sided case~\eqref{eq:ineq:linear}, however we emphasize that they are different in how the $\Gamma$ depends on $\vv{x},\vv{\lambda}$ (i.e. the difference between Lemma~\ref{lem:twosided:gamma} and Lemma~\ref{lem:ineq:gamma_def}). We define $P$ the same as the one sided case (\ref{eq:ineq:P}),
\begin{align}
P = \left[\begin{array}{cc}
\eta c I & \eta A^T\\
\eta A & c I
\end{array}  \right] \label{eq:twosided:P}
\end{align} 
where $c = 20 \ell [ \max(\frac{\rho \kappa_2}{\mu}, \frac{\ell}{\mu})]^2 [\max(\frac{\eta}{\ell \rho}, \frac{\ell}{\mu})]^2 \frac{\kappa_2}{\kappa_1}$, and Lyapunov function $V({\vv{z}}) = ({\vv{z}}  - \vv{z}^*)^T P ({\vv{z}} - \vv{z}^*)$.  
Then we claim that an identical version of Lemma~\ref{lem:ineq:qp}, stated below as Lemma~\ref{lem:twosided:qp}, still holds. The reason is that, the Aug-PDGD-TS (\ref{eq:twosided:linear}) differs from Aug-PDGD (\ref{eq:ineq:linear}) only in how $\Gamma$ depends on $\vv{x},\vv{\lambda}$. However, the only property regarding $\Gamma(\vv{z})$ used in the proof of Lemma~\ref{lem:ineq:qp} is that $\Gamma$ is a diagonal matrix with each diagonal entry lying within $[0,1]$, which still holds in Aug-PDGD-TS (\ref{eq:twosided:linear}). Therefore, Lemma~\ref{lem:twosided:qp} automatically holds and its proof is exactly the same as that of Lemma~\ref{lem:ineq:qp}. 
\begin{lemma}\label{lem:twosided:qp}
	Regardless of the value of $\vv{z} $, we have $-G(\vv{z})^T P - P G(\vv{z})\succeq \tau P $, where $\tau = \frac{\eta \kappa_1}{2c} = \frac{\eta\kappa_1^2}{40 \ell \kappa_2 [ \max(\frac{\rho \kappa_2}{\mu}, \frac{\ell}{\mu})]^2 [\max(\frac{\eta}{\ell \rho}, \frac{\ell}{\mu})]^2 }$.
\end{lemma}
With Lemma~\ref{lem:twosided:qp} we can prove the exponential decay of Lyapunov function $V({\vv{z}})$,
\begin{align*}
\frac{d}{dt} V({\vv{z}}) = (\vv{z} - \vv{z}^*)^T(G(\vv{z})^T P + P G(\vv{z})) (\vv{z}- \vv{z}^*)\leq (\vv{z} - \vv{z}^*)^T(-\tau P) (\vv{z}- \vv{z}^*) = -\tau V(\vv{z})\
\end{align*}
 and hence finish the proof of Theorem~\ref{thm:twosided:stablility}.

\subsection{Relaxing the Rank Constraint}\label{appendix:rank}
In this section, we relax the rank constraint (Assumption~\ref{assump:A}). The content of this section is self-complete and does not depend on the main text of this paper. We consider the following optimization problem,
	\begin{align}
	\min_{\vv{x}\in \R^n} \quad & f(\vv{x}) \label{eq:rank:opt_problem} \\
	\text{s.t.} \quad&  A\vv{x} \leq \vv{b} \nonumber
	\end{align}
	where $f$ satisfies Assumption~\ref{assump:rank:f}, $\vv{b} = [b_1,\ldots,b_m]^T$ and $A^T = [\vv{a}_1,\vv{a}_2,\ldots, \vv{a}_m]$, with each $\vv{a}_j\in\R^n$. 
	\begin{assumption}\label{assump:rank:f}
		Function $f$ is twice differentiable, $\mu$-strongly convex and $\ell$-smooth, i.e. for all $\vv{x},\vv{y}\in\R^n$, 
		\begin{align}
		\mu \Vert \vv{x} - \vv{y}\Vert^2 \leq \langle \nabla f(\vv{x}) - \nabla f(\vv{y}) , \vv{x} - \vv{y}\rangle \leq \ell \Vert \vv{x} - \vv{y}\Vert^2
		\end{align}
	\end{assumption}
	We further make the following assumption, 
	\begin{assumption}\label{assump:rank:A}
		The optimization problem \eqref{eq:rank:opt_problem} has a unique minimizer $\vv{x}^*$ and satisfies the linear independent constraint qualification. In details, without loss of generality we let at $\vv{x}^*$, we let the first $m_1$ constraints be active ($\vv{a}_j^T \vv{x} = b_j$ for $j=1,\ldots,m_1$), while the rest $m_2 = m - m_1$ constraints be inactive ($\vv{a}_j^T \vv{x} < b_j$ for $j=m_1+1,\ldots,m$ ). Partition $A$ into $A = [A_1^T, A_2^T]^T$, where $A_1 = [\vv{a}_1,\vv{a}_2,\ldots,\vv{a}_{m_1}]^T$ is the first $m_1$ rows of $A$ and $A_2 = [\vv{a}_{m_1+1},\ldots,\vv{a}_m]^T$ is the $(m_1+1)$th through the last row. We make the following assumptions.
		\begin{enumerate}
			\item[(a)] Matrix $A_1$ has full row rank, with $\kappa_1 I\preceq A_1A_1^T$ for some $\kappa_1>0$.
			\item[(b)] $A A^T \preceq \kappa_2 I$. 
			\item[(c)] There exists $\epsilon>0$ s.t. for any $j>m_1$,  $  \vv{a}_j^T\vv{x}^* - b_j\leq -\epsilon$. 
		\end{enumerate} 
		\end{assumption}
	
Then we define the augmented Lagrangian, 
	\begin{align}
		L  (\vv{x},\vv{\lambda})  = f(\vv{x} ) + \sum_{j=1}^m H_\rho(\vv{a}_j^T \vv{x}- b_j, \lambda_j ) \label{eq:rank:aug_lagrangian}
		\end{align} where $\rho>0$ is a free parameter, $H_\rho(\cdot,\cdot):\R^2\rightarrow \R$ is a penalty function on constraint violation, defined as follows 
	\begin{align*}
		H_\rho(\vv{a}_j^T \vv{x}- b_j , \lambda_j) 
		&= \left\{ \begin{array}{ll}
		(\vv{a}_j^T \vv{x}- b_j ) \lambda_j   + \frac{\rho}{2}(\vv{a}_j^T \vv{x}- b_j)^2  & \text{if } \rho(\vv{a}_j^T \vv{x}- b_j) + \lambda_j\geq 0\\
		-\frac{1}{2} \frac{\lambda_j^2}{\rho}& \text{if } \rho(\vv{a}_j^T \vv{x}- b_j) + \lambda_j < 0
		\end{array}  \right.
		\end{align*}
		We can then calculate the gradient of $H_\rho$ w.r.t. $\vv{x}$ and $\vv{\lambda}$.
	{\small\begin{align*}
		&\nabla_{\vv{x}} H_\rho(\vv{a}_j^T\vv{x} - b_j , \lambda_j)= \max(\rho(\vv{a}_j^T\vv{x} - b_j) + \lambda_j,0) \vv{a}_j \\
		&\nabla_{\vv{\lambda}} H_\rho(\vv{a}_j^T\vv{x} - b_j , \lambda_j) = \frac{ \max(\rho(\vv{a}_j^T\vv{x} - b_j) + \lambda_j,0) - \lambda_j}{\rho} \vv{e}_j 
		\end{align*}}where $\vv{e}_j \in\R^m$ is a vector with the $j$'th entry being $1$ and other entries being $0$. The primal-dual gradient dynamics for the augmented Lagrangian $L $ is given in \eqref{eq:rank:pdgd}. 
	\begin{subeqnarray}\label{eq:rank:pdgd}
			\dot{\vv{x}} &=& - \nabla_{\vv{x}}L  (\vv{x},\vv{\lambda}) = -\nabla f(\vv{x}) - \sum_{j=1}^m \nabla_{\vv{x}} H_\rho(\vv{a}_j^T\vv{x} - b_j , \lambda_j) \nonumber\\
			&=&-\nabla f(\vv{x}) - \sum_{j=1}^m \max(\rho(\vv{a}_j^T\vv{x} - b_j) + \lambda_j,0) \vv{a}_j \slabel{eq:rank:pdgd_p}\\
			\dot{\vv{\lambda}} &= &\eta \nabla_{\vv{\lambda}}L  (\vv{x},\vv{\lambda}) = \eta \sum_{j=1}^{m} \nabla_{\vv{\lambda}} H_\rho(\vv{a}_j^T\vv{x} - b_j , \lambda_j)\nonumber\\\
			&=& \eta \sum_{j=1}^{m} \frac{ \max(\rho(\vv{a}_j^T\vv{x} - b_j) + \lambda_j,0) - \lambda_j}{\rho} \vv{e}_j  \slabel{eq:rank:pdgd_d}
	\end{subeqnarray}

Then, we define, $\vv{z} = [ \vv{x}^T,\vv{\lambda}  ^T]^T$ and $\vv{z}^*$ similarly. We have the following Theorem showing the exponential stability of dynamics \eqref{eq:rank:pdgd}. 
\begin{theorem}\label{thm:rank:convergence}
	Under Assumption~\ref{assump:rank:f} and \ref{assump:rank:A}, dynamics \eqref{eq:rank:pdgd} is exponentially stable in the sense that, $$\Vert \vv{z} - \vv{z}^*\Vert = O(e^{-\tau t})$$ where $\tau = \frac{\eta \kappa_1}{2c}$, and $c$ is the smallest positive constant satisfying
	\begin{align*}
	c &\geq \rho\kappa_2\\
	\frac{1}{2} \eta \kappa_1[\frac{2\eta c}{\rho} (1 - \bar{\gamma})  - 2\eta \kappa_2 -\eta \kappa_1] &\geq (2\eta \kappa_2)^2\\
	2\eta c \mu - \eta^2 \kappa_2  - \frac{\eta^2\kappa_1}{2}  &\geq \frac{2\kappa_2}{\eta \kappa_1} ( \eta\ell   + \eta \rho (\kappa_2) + \frac{\eta^2}{\rho} + \frac{\eta^2\kappa_1}{2c} )^2 
	\end{align*}
	where $\bar{\gamma} = \frac{\xi(\vv{z}(0))}{\xi(\vv{z}(0))+\rho\epsilon}\in [0,1)$ and $\xi (\vv{z}(0))$ is defined as,
$$	\xi (\vv{z}(0)) =\max_{i=1,\ldots,n}\Big\{(\rho \Vert \vv{a}_i\Vert + \sqrt{\eta}) \sqrt{  \Vert \vv{x}(0) - \vv{x}^*\Vert^2 + \frac{1}{  \eta} \Vert \vv{\lambda}(0) - \vv{\lambda}^*\Vert^2  }+\rho\Vert \vv{a}_i\Vert \Vert \vv{x}^*\Vert  + \rho |b_i| + \Vert  \vv{\lambda}^*\Vert \Big\}.$$
\end{theorem}

\noindent\textit{Proof of Theorem~\ref{thm:rank:convergence}: } We start with the following Lemma.
\begin{lemma}\label{lem:rank:gamma}
	For any $j$ and $\vv{z} = [\vv{x}^T,\vv{\lambda}^T]^T\in\R^{n+m}$, we define $$\gamma_j(\vv{z}) = \left\{ \begin{array}{ll}
\frac{\max( \rho(\vv{a}_j^T\vv{x} - b_j)+\lambda_j,0) - \max( \rho(\vv{a}_j^T\vv{x}^* - b_j)+\lambda_j^*,0) }{\rho \vv{a}_j^T(\vv{x} - \vv{x}^*) + \lambda_j - \lambda_j^* } & \text{if }\rho \vv{a}_j^T(\vv{x} - \vv{x}^*) + \lambda_j - \lambda_j^* \neq 0\\
0 & \text{if }\vv{a}_j^T(\vv{x} - \vv{x}^*) + \lambda_j - \lambda_j^* = 0
	\end{array} \right.$$
and further define $\Gamma(\vv{z}) = \text{diag}(\gamma_1(\vv{z}),\ldots,\gamma_m(\vv{z}))$. Then, $\gamma_j(\vv{z} )\in[0,1]$ and 
	{\small\begin{align*}
		\nabla_{\vv{x}} H_\rho(\vv{a}_j^T\vv{x} - b_j , \lambda_j) - \nabla_{\vv{x}} H_\rho(\vv{a}_j^T\vv{x}^* - b_j , \lambda_j^*)
		& = \gamma_j(\vv{z})  \rho \vv{a}_j^T(\vv{x}- \vv{x}^*) \vv{a}_j+ \gamma_j(\vv{z})( \lambda_j- \lambda^*_j)\vv{a}_j  \\
		\nabla_{\vv{\lambda}} H_\rho(\vv{a}_j^T\vv{x} - b_j , \lambda_j) - \nabla_{\vv{\lambda}} H_\rho(\vv{a}_j^T\vv{x}^* - b_j , \lambda_j^*) 
		&= \gamma_j (\vv{z})\vv{a}_j^T(\vv{x} - \vv{x}^*) \vv{e}_j + \frac{1}{\rho} (\gamma_j (\vv{z}) - 1) (\lambda_j - \lambda^*_j) \vv{e}_j
		\vspace{-6pt}
		\end{align*}}
	\vspace{-15pt}
	\begin{proof}
		The lemma directly follows from that for any $y, y^*\in\R$, there exists some $\gamma\in[0,1]$, depending on $y,y^*$ s.t. $\max(y,0) - \max(y^*,0) = \gamma (y - y^*)$. To see this, when $y \neq y^*$, set $\gamma = \frac{\max(y,0) - \max(y^*,0) }{y - y^*} $; otherwise, set $\gamma= 0$.  	
	\end{proof}
\end{lemma}
Utilizing the definition of $\gamma_j(\vv{z})$, we can rewrite \eqref{eq:rank:pdgd},
{\small
	\begin{align*}
	\dot{\vv{x}} &=- ( \nabla_{\vv{x}} L (\vv{x},\vv{\lambda}) - \nabla_{\vv{x}} L  (\vv{x}^*, \vv{\lambda}^*) ) \\
	&=  - (\nabla f(\vv{x}) - \nabla f(\vv{x}^*) )  -  \sum_{j=1}^m \Bigg[\nabla_{\vv{x}} H_\rho(\vv{a}_j^T\vv{x} - b_j , \lambda_j) - \nabla_{\vv{x}} H_\rho(\vv{a}_j^T\vv{x}^* - b_j , \lambda_j^*) \Bigg]\\
	&= - B(\vv{x})(\vv{x} -\vv{x}^*) - \sum_{j=1}^m  \rho \gamma_j(\vv{z}) \vv{a}_j^T(\vv{x} - \vv{x}^*) \vv{a}_j   - \sum_{j=1}^m \gamma_j(\vv{z}) (\lambda_j - \lambda^*_j)\vv{a}_j\\
	&= - B(\vv{x})(\vv{x} - \vv{x}^*) - \rho A^T\Gamma(\vv{z}) A(\vv{x}- \vv{x}^*) - A^T \Gamma(\vv{z})(\vv{\lambda} - \vv{\lambda}^*) 
	\end{align*} }where $\mu I \preceq B(\vv{x}) \preceq \ell I$ (Lemma~\ref{lem:B}). Similarly, rewrite \eqref{eq:rank:pdgd_d}, 
{\small\begin{align*}
	\dot{\vv{\lambda}} &= \eta \nabla_{\vv{\lambda}} L (\vv{x},\vv{\lambda}) - \eta \nabla_{\vv{\lambda}} L (\vv{x}^*,\vv{\lambda}^*) \\
	&= \eta \sum_{j=1}^m\Big[ \nabla_{\vv{\lambda}} H_\rho(\vv{a}_j^T\vv{x} - b_j,\lambda_j ) - \nabla_{\vv{\lambda}} H_\rho(a_j^T\vv{x}^* - b_j, {\lambda}_j^* )\Big]\\
	&= \eta \sum_{j=1}^m \Big[ \gamma_j(\vv{z}) \vv{a}_j^T(\vv{x} - \vv{x}^*) \vv{e}_j +\frac{1}{\rho} (\gamma_j(\vv{z}) - 1) (\lambda_j - \lambda_j^*) \vv{e}_j\Big]\\
	&= \eta \Gamma(\vv{z}) A (\vv{x}-\vv{x}^*) + \frac{\eta}{\rho} (\Gamma(\vv{z}) - I) ( \vv{\lambda} -\vv{\lambda}^*)
	\end{align*}}In summary, the dynamics \eqref{eq:rank:pdgd} can be written as,

\begin{align}
	\frac{d}{dt}  \vv{z}   = \left[\begin{array}{cc}
		-B({\vv{x}}) - \rho A^T\Gamma(\vv{z}) A & - A^T\Gamma(\vv{z})\\
		\eta\Gamma(\vv{z}) A &  \frac{\eta}{\rho}(\Gamma(\vv{z}) - I)
	\end{array}\right] (\vv{z} - \vv{z}^*)  \coloneqq G(\vv{z}) (\vv{z} - \vv{z}^*). \label{eq:rank:linear}
\end{align}
Then, we define matrix $P$ as, 
\begin{align}
P = \left[\begin{array}{cc}
\eta c I & \eta A^T\\
\eta A & c I
\end{array}  \right]. \label{eq:rank:P}
\end{align} 
With our choice of $c$ in Theorem~\ref{thm:rank:convergence}, it is easy to check $P$ is positive definite. We then define Lyapunov function
\begin{align}
V(\vv{z}) = \frac{1}{2} (\vv{z} - \vv{z}^*)^T P (\vv{z} - \vv{z})
\end{align}
Then, we calculate the derivative (w.r.t. time $t$) of $V(\vv{z}) $,
\begin{align}
&\frac{d}{dt} V(\vv{z})=  \dot{\vv{z}}^T P ( \vv{z}-\vv{z}^*) + (\vv{z}-\vv{z}^*)^T P \dot{\vv{z}}\nonumber \\
&=(\vv{z}-\vv{z}^*) ^T(G(\vv{z})^TP + PG(\vv{z}))(\vv{z} - \vv{z}^*) \label{eq:rank:dvgp}
\end{align}
The following Lemma is critical in establishing the exponential decay of the Lyapunov function $V(\vv{z})$.
\begin{lemma}\label{lem:rank:qp}
	Recall that constant $\tau$ is defined as $\tau = \frac{\eta \kappa_1}{2c}$. We have for any $\vv{z}\in\R^{n+m}$, $- G(\vv{z})^T P - P G(\vv{z}) \succeq \tau P$, 
\end{lemma}
With Lemma~\ref{lem:rank:qp}, we have, continuing \eqref{eq:rank:dvgp},
$$ \frac{d}{dt} V(\vv{z}) \leq -\tau V(\vv{z})$$
and hence, $V(\vv{z}) \leq e^{-\tau t} V(\vv{z}(0))$, hence the conclusion of Theorem~\ref{thm:rank:convergence}. The only part that is left to prove is Lemma~\ref{lem:rank:qp}, which will be done in Appendix-\ref{subsec:rank:qp}.\qedd

\subsection{Proof of Lemma~\ref{lem:rank:qp}}  \label{subsec:rank:qp}
Our first result is boundedness of the $\rho(\vv{a}_j^T \vv{x} - b_j) + \lambda_j$ along its trajectory.  
\begin{lemma}
	We have for all $j$, $|\rho(\vv{a}_j^T \vv{x} - b_j) + \lambda_j|< \xi(\vv{z}(0))$, where we recall that $\xi(\vv{z}(0))$, defined in Theorem~\ref{thm:rank:convergence}, is a constant that depends on the initial value of the dynamics given by, 
	$$ \xi (\vv{z}(0)) =\max_{i=1,\ldots,n}\Big\{(\rho \Vert \vv{a}_i\Vert + \sqrt{\eta}) \sqrt{  \Vert \vv{x}(0) - \vv{x}^*\Vert^2 + \frac{1}{  \eta} \Vert \vv{\lambda}(0) - \vv{\lambda}^*\Vert^2  }+\rho\Vert \vv{a}_i\Vert \Vert \vv{x}^*\Vert  + \rho |b_i| + \Vert  \vv{\lambda}^*\Vert \Big\} $$
\end{lemma}
\begin{proof}
	We construct Lyapunov function, 
	$$V_0(\vv{x},\vv{\lambda}) =\frac{1}{2} \Vert \vv{x} - \vv{x}^*\Vert^2 + \frac{1}{2 \eta} \Vert \vv{\lambda} - \vv{\lambda}^*\Vert^2$$
	Then, 
	\begin{align*}
		\frac{d}{dt} V_0(\vv{x},\vv{\lambda}) &= \langle \vv{x} - \vv{x}^*, -\nabla_{\vv{x}}L (\vv{x},\vv{\lambda}) \rangle + \langle \vv{\lambda} - \vv{\lambda}^*,\nabla_{\vv{\lambda}}L  (\vv{x},\vv{\lambda}) \rangle\\
		&\leq L(\vv{x}^*,\vv{\lambda}) - L(\vv{x},\vv{\lambda}) + L(\vv{x},\vv{\lambda}) - L(\vv{x},\vv{\lambda}^*)\\
		&= L(\vv{x}^*,\vv{\lambda}) - L(\vv{x}^*,\vv{\lambda}^*) + L(\vv{x}^*,\vv{\lambda}^*) - L(\vv{x},\vv{\lambda}^*)\\
		&\leq 0
	\end{align*}
	where the first inequality is due to that $L$ is convex in $\vv{x}$, and concave in $\vv{\lambda}$, and the second inequality is due to that $(\vv{x}^*$, $\vv{\lambda}^*)$ is the saddle point of $L$. The above display immediately implies  
	$$V_0(\vv{x},\vv{\lambda}) \leq V_0(\vv{x}(0),\vv{\lambda}(0))$$
	Hence, for any $j$,
	\begin{align*}
		|\rho(\vv{a}_j^T \vv{x} - b_j) + \lambda_j| &\leq \rho \Vert \vv{a}_j \Vert \Vert \vv{x}\Vert + \rho |b_j| + |\lambda_j| \\
		&\leq \rho \Vert \vv{a}_j \Vert \Vert \vv{x} - \vv{x}^* \Vert + \rho\Vert \vv{a}_j\Vert \Vert \vv{x}^*\Vert  + \rho |b_j| + \Vert \vv{\lambda} - \vv{\lambda}^*\Vert +\Vert  \vv{\lambda}^*\Vert\\
		&\leq  \rho \Vert \vv{a}_j \Vert \sqrt{2 V_0(\vv{x}(0),\vv{\lambda}(0)) }+ \rho\Vert \vv{a}_j\Vert \Vert \vv{x}^*\Vert  + \rho |b_j| + \sqrt{2\eta V_0(\vv{x}(0),\vv{\lambda}(0))  } +\Vert  \vv{\lambda}^*\Vert\\
		&\leq (\rho \Vert \vv{a}_j\Vert + \sqrt{\eta}) \sqrt{  \Vert \vv{x}(0) - \vv{x}^*\Vert^2 + \frac{1}{  \eta} \Vert \vv{\lambda}(0) - \vv{\lambda}^*\Vert^2  }+\rho\Vert \vv{a}_j\Vert \Vert \vv{x}^*\Vert  + \rho |b_j| + \Vert  \vv{\lambda}^*\Vert \\
		&\leq \xi(\vv{z}(0)).
	\end{align*}
\end{proof}
A direct corollary of the above Lemma is that, for $j>m_1$, $\gamma_j(\vv{z})$ is bounded away from $1$. 
\begin{corollary}\label{cor:rank:gammamax}
	For $j>m_1$, we have $\gamma_j(\vv{z})\leq \bar{\gamma} \coloneqq \frac{\xi(\vv{z}(0))}{\xi(\vv{z}(0)) + \rho\epsilon} <1$.
\end{corollary}
\begin{proof}
	Recall that the definition of $\gamma_j(\vv{z})$ is $$\gamma_j(\vv{z}) = \left\{ \begin{array}{ll}
		\frac{\max( \rho(\vv{a}_j^T\vv{x} - b_j)+\lambda_j,0) - \max( \rho(\vv{a}_j^T\vv{x}^* - b_j)+\lambda_j^*,0) }{\rho \vv{a}_j^T(\vv{x} - \vv{x}^*) + \lambda_j - \lambda_j^* } & \text{if }\rho \vv{a}_j^T(\vv{x} - \vv{x}^*) + \lambda_j - \lambda_j^* \neq 0\\
		0 & \text{if }\vv{a}_j^T(\vv{x} - \vv{x}^*) + \lambda_j - \lambda_j^* = 0
	\end{array} \right.$$
Now for $j>m_1$, we have $\vv{a}_j^T \vv{x}^* - b_j \leq -\epsilon < 0	$ by the definition of $\epsilon$ in Assumption~\ref{assump:rank:A}(c). Also $\lambda_j^* = 0$ since the $j$'th constraint is inactive at $\vv{x}^*$. Therefore, 
$$ \rho(\vv{a}_j^T\vv{x}^* - b_j)+\lambda_j^* \leq -\rho \epsilon < 0$$
and hence, 
\begin{align}
\gamma_j(\vv{z}) & = 
\frac{\max( \rho(\vv{a}_j^T\vv{x} - b_j)+\lambda_j,0) - \max( \rho(\vv{a}_j^T\vv{x}^* - b_j)+\lambda_j^*,0) }{\rho \vv{a}_j^T(\vv{x} - \vv{x}^*) + \lambda_j - \lambda_j^* } \nonumber  \\
&= \frac{\max( \rho(\vv{a}_j^T\vv{x} - b_j)+\lambda_j,0) - \max( \rho(\vv{a}_j^T\vv{x}^* - b_j)+\lambda_j^*,0) }{[\rho(\vv{a}_j^T\vv{x} - b_j)+\lambda_j] - [\rho(\vv{a}_j^T\vv{x}^* - b_j)+\lambda_j^*] } \nonumber \\
&= \frac{\max( \rho(\vv{a}_j^T\vv{x} - b_j)+\lambda_j,0)  }{[\rho(\vv{a}_j^T\vv{x} - b_j)+\lambda_j] - [\rho(\vv{a}_j^T\vv{x}^* - b_j)+\lambda_j^*] } \nonumber \\
&\leq \frac{\max( \rho(\vv{a}_j^T\vv{x} - b_j)+\lambda_j,0)  }{[\rho(\vv{a}_j^T\vv{x} - b_j)+\lambda_j] + \rho\epsilon }\nonumber \\
&\leq \frac{| \rho(\vv{a}_j^T\vv{x} - b_j)+\lambda_j|  }{|\rho(\vv{a}_j^T\vv{x} - b_j)+\lambda_j| + \rho\epsilon }\nonumber \\
&\leq \frac{\xi(\vv{z}(0))}{\xi(\vv{z}(0)) + \rho\epsilon}
\end{align}

%

\end{proof}
We now proceed to prove Lemma~\ref{lem:rank:qp}.

\noindent\textit{Proof of Lemma~\ref{lem:rank:qp}}: Recall that Lemma~\ref{lem:rank:qp} states for any $\vv{z}\in\R^{n+m}$, $- G(\vv{z})^T P - P G(\vv{z}) \succeq \tau P$. We define
	$$\tilde{Q} = - G(\vv{z})^T P - P G(\vv{z}) - \tau P =\left[ \begin{array}{cc}
	Q_1 & Q_3\\
	Q_3^T& Q_2
	\end{array}\right] $$
	In what follows, we will simply write $B(\vv{x})$ and $ \Gamma(\vv{z}) = \text{diag}(\gamma_1(\vv{z}),\ldots,\gamma_m(\vv{z}))$ as $B$ and $\Gamma =\text{diag}(\gamma_1,\ldots,\gamma_m) $, dropping the dependence on $\vv{x}$ and $\vv{z}$, and allow $B$ to be \textit{any} positive definite matrix satisfying $\mu I \preceq B \preceq \ell I$, and $\Gamma$ to be \textit{any} diagonal matrix with each entry bounded in $[0,1]$, with the additional constraint that for $j>m_1$, $\gamma_j \leq\bar{\gamma}$ (see Corollary~\ref{cor:rank:gammamax}).  After straightforward calculations, we have
	\begin{align*}
		Q_1 &= 2\eta c B  + 2(\eta c\rho - \eta^2) A^T\Gamma A - \tau\eta c I \\ 
		Q_2 & = \eta (\Gamma  A A^T + AA^T\Gamma ) + \frac{2\eta c}{\rho} (I - \Gamma  ) - \tau c I \\
		Q_3 & = \eta B A^T + \eta \rho A^T\Gamma A A^T + \frac{\eta^2}{\rho} A^T(I - \Gamma) -  \tau\eta A^T 
	\end{align*}
	
	\begin{lemma}\label{lem:rank:Q2}
		When $c$ is large enough s.t. $\frac{1}{2} \eta \kappa_1[\frac{2\eta c}{\rho} (1 - \bar{\gamma})  - 2\eta \kappa_2 -\eta \kappa_1] \geq (2\eta \kappa_2)^2$ and $c \geq \kappa_2\rho$, we have matrix $Q_2$ is lower bounded by $Q_2\succeq \frac{\eta \kappa_1 }{2}  I $.
	\end{lemma}
	\begin{proof}
		Notice that $A = [A_1^T, A_2^T]^T$, where $A_1$ consists of the first $m_1$ rows of $A$ while $A_2$ is the $m_1+1$ through the last row of $A$. Then, 
		\begin{align}
		AA^T = \left[\begin{array}{cc}
			A_1 A_1^T & A_1 A_2^T\\
			A_2 A_1^T & A_2 A_2^T
		\end{array}\right]
		\end{align}
		We also divide $\Gamma$ into $\Gamma_1=\text{diag}(\gamma_1,\ldots,\gamma_{m_1}) $ and $\Gamma_2=\text{diag}(\gamma_{m_1+1},\ldots,\gamma_m)$. Then,
		\begin{align}
		Q_2 = \left[ \begin{array}{cc}
		\eta (\Gamma_1 A_1 A_1^T + A_1 A_1^T\Gamma_1  ) + \frac{2\eta c}{\rho} (I - \Gamma_1 ) & \eta (\Gamma_1 A_1 A_2^T + A_1 A_2^T \Gamma_2) \\
		\eta (\Gamma_2 A_2 A_1^T + A_2 A_1^T \Gamma_1 ) & \eta(\Gamma_2 A_2 A_2^T + A_2A_2^T\Gamma_2) + \frac{2\eta c}{\rho}(I - \Gamma_2)
 		\end{array}\right] - \tau c I \label{eq:rank:q2}
		\end{align} 
	We will then lower bound $\eta (\Gamma_1 A_1 A_1^T + A_1 A_1^T\Gamma_1  ) + \frac{2\eta c}{\rho} (I - \Gamma_1 )$ and $\eta(\Gamma_2 A_2 A_2^T + A_2A_2^T\Gamma_2) + \frac{2\eta c}{\rho}(I - \Gamma_2)$. Bounding the latter is easy, since by Corollary~\ref{cor:rank:gammamax}, we have $I - \Gamma_2 \succeq (1 - \bar{\gamma} )I $, and hence
	\begin{align}
	\eta(\Gamma_2 A_2 A_2^T + A_2A_2^T\Gamma_2) + \frac{2\eta c}{\rho}(I - \Gamma_2) \succeq [\frac{2\eta c}{\rho} (1 - \bar{\gamma})  - 2\eta \kappa_2 ] I \label{eq:rank:q2_lb1}
	\end{align}
	Next, we show the following lower bound
	\begin{align}
\eta (\Gamma_1 A_1 A_1^T + A_1 A_1^T\Gamma_1  ) + \frac{2\eta c}{\rho} (I - \Gamma_1 )\succeq \frac{3}{2} \eta A_1 A_1^T \succeq \frac{3}{2} \eta \kappa_1 I.	\label{eq:rank:q2_lb2}
	\end{align}


	Noticing that $\Gamma_1 = \text{diag}(\gamma_1,\ldots,\gamma_{m_1})$ and each $\gamma_j \in [0,1]$, we have $\eta (\Gamma_1 A_1 A_1^T + A_1 A_1^T\Gamma_1  ) + \frac{2\eta c}{\rho} (I - \Gamma_1 )$ is a convex combination of matrices $\{ \eta (\Gamma_1 A_1 A_1^T + A_1 A_1^T\Gamma_1  ) + \frac{2\eta c}{\rho} (I - \Gamma_1 ): \gamma_j = 0 \text{ or } 1 \}$. Therefore, without loss of generality, to prove \eqref{eq:rank:q2_lb2} we only need to prove 
	\begin{align}
	\eta (\Gamma_1^k A_1 A_1^T + A_1 A_1^T\Gamma_1^k  ) + \frac{2\eta c}{\rho} (I - \Gamma_1^k ) \succeq \frac{3}{2} \eta A_1 A_1^T, \quad \forall k=0,\ldots,m_1 \label{eq:rank:q2_lb3}
	\end{align}
	where $\Gamma_1^k$ is defined as
	$$\Gamma_1^k = \text{diag}(\overbrace{1,\ldots,1,}^{k \text{ entries}}0,\ldots,0) $$
Now, we write $A_1 A_1^T$ in block diagonal form
\begin{align}
A_1 A_1^T = \left[\begin{array}{cc}
\Lambda_1 & \Lambda_3\\
\Lambda_3^T & \Lambda_2
\end{array}  \right] \nonumber
\end{align}
where $\Lambda_1$ is $k$-by-$k$, $\Lambda_2$ is $(m_1-k)$-by-$(m_1-k)$, and $\Lambda_3$ is $k$-by-$(m_1 - k)$. Then,
\begin{align}
 \eta (\Gamma_1^k A_1 A_1^T + A_1 A_1^T\Gamma_1^k  ) + \frac{2\eta c}{\rho} (I - \Gamma_1^k ) &=  \left[\begin{array}{cc}
2\eta\Lambda_1 & \eta\Lambda_3\\
\eta\Lambda_3^T &  \frac{2\eta c}{\rho} I
 \end{array}\right] \succeq  \left[\begin{array}{cc}
 2\eta\Lambda_1 & \eta\Lambda_3\\
 \eta\Lambda_3^T &  2\eta \Lambda_2
 \end{array}\right] \\
 &\succeq  \frac{3}{2} \left[\begin{array}{cc}
 \eta\Lambda_1 & \eta\Lambda_3\\
 \eta\Lambda_3^T &  \eta \Lambda_2
 \end{array}\right]  = \frac{3}{2} \eta A_1 A_1^T
\end{align}

Hence \eqref{eq:rank:q2_lb3} is proven, and therefore lower bound \eqref{eq:rank:q2_lb2} is true.

With lower bound \eqref{eq:rank:q2_lb1} and \eqref{eq:rank:q2_lb2} we can now proceed to lower bounding $Q_2$ \eqref{eq:rank:q2},
\begin{align}
Q_2 &\succeq \left[ \begin{array}{cc}
\frac{3}{2}\eta \kappa_1 I & \eta (\Gamma_1 A_1 A_2^T + A_1 A_2^T \Gamma_2) \\
\eta (\Gamma_2 A_2 A_1^T + A_2 A_1^T \Gamma_1 ) & [\frac{2\eta c}{\rho} (1 - \bar{\gamma})  - 2\eta \kappa_2 ] I  
\end{array}\right] - \frac{\eta \kappa_1 }{2}  I \nonumber \\
&\succeq  \eta \kappa_1 I - \frac{\eta \kappa_1 }{2}  I = \frac{\eta \kappa_1 }{2}  I .
\end{align}  
where in the last inequality, we have required $c$ to be large enough ($\frac{1}{2} \eta \kappa_1[\frac{2\eta c}{\rho} (1 - \bar{\gamma})  - 2\eta \kappa_2 -\eta \kappa_1] \geq (2\eta \kappa_2)^2$ ) s.t. 
$$ \left[ \begin{array}{cc}
\frac{3}{2}\eta \kappa_1 I & \eta (\Gamma_1 A_1 A_2^T + A_1 A_2^T \Gamma_2) \\
\eta (\Gamma_2 A_2 A_1^T + A_2 A_1^T \Gamma_1 ) & [\frac{2\eta c}{\rho} (1 - \bar{\gamma})  - 2\eta \kappa_2 ] I  
\end{array}\right]  \succeq \eta \kappa_1 I $$
\end{proof}
With the lower bound on $Q_2$, the rest is to employ Schur Complement to show $\tilde{Q}\succeq 0$. We need to prove
\begin{align}
Q_2 &\succ 0 \label{eq:rank:schur1} \\
Q_1 - Q_3 Q_2^{-1} Q_3^T & \succeq 0\label{eq:rank:schur2}
\end{align}
Eq. \eqref{eq:rank:schur1} is already shown by Lemma~\ref{lem:rank:Q2}. To show \eqref{eq:rank:schur2}, we have
\begin{align}
Q_1 \succeq 2\eta c \mu I - \eta^2 \kappa_2 I - \frac{\eta^2\kappa_1}{2} I
\end{align}
and
\begin{align*}
Q_3 Q_2^{-1} Q_3^T &\preceq \frac{2}{\eta \kappa_1} Q_3 Q_3^T \preceq \frac{2}{\eta \kappa_1} \Vert Q_3\Vert^2 I \\
&\preceq \frac{2}{\eta \kappa_1} ( \eta\ell \sqrt{\kappa_2}  + \eta \rho (\kappa_2)^{3/2} + \frac{\eta^2}{\rho} \sqrt{\kappa_2}+ \frac{\eta^2\kappa_1}{2c} \sqrt{\kappa_2})^2 I\\
&=  \frac{2\kappa_2}{\eta \kappa_1} ( \eta\ell   + \eta \rho (\kappa_2) + \frac{\eta^2}{\rho} + \frac{\eta^2\kappa_1}{2c} )^2 I
\end{align*}
So as long as $c$ is large enough s.t. 
$$2\eta c \mu - \eta^2 \kappa_2  - \frac{\eta^2\kappa_1}{2}  \geq \frac{2\kappa_2}{\eta \kappa_1} ( \eta\ell   + \eta \rho (\kappa_2) + \frac{\eta^2}{\rho} + \frac{\eta^2\kappa_1}{2c} )^2  $$
we have \eqref{eq:rank:schur2} is true, and hence $\tilde{Q}\succeq 0$.

}{}

\end{document}